\documentclass{amsart}
\usepackage{amsmath,mathtools,amssymb}
\usepackage{amsrefs}
\usepackage{microtype}
\usepackage{lmodern}
\usepackage{hyperref}
\usepackage{tikz-cd}
\usepackage{graphicx}
\usepackage{tikz-network}
\usetikzlibrary{arrows}
\usepackage[thinlines,thiklines]{easybmat}
\usepackage{enumerate}
 
\usetikzlibrary{decorations.markings}
\usepackage{caption}
\theoremstyle{plain}
\newtheorem{theorem}[subsection]{Theorem}
\newtheorem{proposition}[subsection]{Proposition}
\newtheorem{lemma}[subsection]{Lemma}
\newtheorem{corollary}[subsection]{Corollary}

\theoremstyle{definition}
\newtheorem{definition}[subsection]{Definition}
\newtheorem{example}[subsection]{Example}

\theoremstyle{remark}
\newtheorem{remark}[subsection]{Remark}

\numberwithin{equation}{section}

\newcommand{\cal}{\mathcal}
\newcommand{\cla}{{\cal A}}
\newcommand{\clb}{{\cal B}}
\newcommand{\clc}{{\cal C}}
\newcommand{\clg}{{\cal G}}
\newcommand{\clv}{{\cal V}}

\newcommand{\ot}{\otimes}
\newcommand{\raro}{\rightarrow}

\begin{document}

\title[]{Quantum isomorphism of 2-graphs}
\author[Joardar]{Soumalya Joardar}
\address{Department of Mathematics and Statistics, 
	Indian Institute of Science Education and Research Kolkata, 
	Mohanpur - 741246, West
	Bengal, India}
\email{soumalya@iiserkol.ac.in}
\author[Rahaman]{Atibur Rahaman}
\address{Department of Mathematics and Statistics, 
	Indian Institute of Science Education and Research Kolkata, 
	Mohanpur - 741246, West
	Bengal, India}
\email{atibur.pdf@iiserkol.ac.in}
\author[Sharma]{Jitender Sharma}
\address{Department of Mathematics and Statistics, 
	Indian Institute of Science Education and Research Kolkata, 
	Mohanpur - 741246, West
	Bengal, India}
\email{js20ip007@iiserkol.ac.in}

\begin{abstract} 
	We formulate a notion of the quantum automorphism group of a $2$-graph. 
	After some preliminary computations, we define quantum isomorphism between 
	a pair of $2$-graphs. We produce a `non-trivial' example of a pair of $2$-graphs 
	that are not quantum isomorphic to each other.   
\end{abstract}



\maketitle

\section{Introduction} 
    Higher rank graphs (synonymously rank \(k\) graphs or \(k\)-graphs) are 
    natural generalization of graphs. Roughly speaking, a $k$-graph $\Lambda$ 
    is a countable small category equipped with a ``degree map" 
    $d:\Lambda\raro\mathbb{N}^{k}$ with the morphisms admitting 
    a unique factorization property. Since its inception, it has played a key role 
    in the realm of noncommutative geometry and topology. 
    In particular, it helps in the proper generalization of Cuntz-Krieger 
    $C^{\ast}$-algebras to higher-rank graph $C^{\ast}$-algebras. 
    There is a formidable literature on higher rank graphs and their 
    $C^{\ast}$-algebras 
    (see \cites{NYJM,Steger1,Steger2,Sims,Sims1,Sims2,Hazzlewood1} to name a few). 
    Pioneered by S. Wang (\cite{Wang}), the quantum symmetry of various finite dimensional 
    or infinite dimensional classical or quantum objects has been quite extensively studied 
    by many authors in recent years (e.g. \cites{Goswami,Banica,Hajac,Skalski,Asfaq,Voigt2,Vaes}). 
    Such quantum symmetry is now well established as the generalized symmetry 
    in the world of noncommutative geometry or topology. Among the finite structures, 
    the quantum automorphism group of finite graphs (see \cites{Schmidt,Fulton}), 
    at least for small number of vertices, are reasonably well understood by now. 
    Moreover, the notion of the quantum automorphism groups of graphs has found its 
    way into quantum information theory (see \cites{Qinfo,Qinfo2}) as well as some 
    concrete physical models involving graphs (see \cite{Potts}). 
    Therefore, higher rank graphs being a generalization of usual graphs, 
    it is quite natural to consider the notion of quantum automorphism groups 
    of higher rank graphs. 
    
    In this paper, we define the notion of ``the'' quantum 
    automorphism group of a rank $2$ graph. The notion is, in a sense, quite natural. 
    Then, with a growing list of literature in the intersection of quantum groups and 
    nonlocal games, we are naturally led to define quantum isomorphism between a 
    pair of $2$-graphs. We restrict our attention to the class of $2$-graphs as they 
    have an equivalent formulation in terms of a simpler combinatorial data 
    circumventing the category theoretic machinery. Briefly speaking, a $2$-graph 
    can be described by a triple consisting of two $1$-graphs on the same vertex 
    set having commuting vertex matrices and a bijection between certain sets. 
    We would like to mention that we restrict our attention to the case where the 
    common vertex set has finitely many vertices to be able to work in the category 
    of compact quantum groups. Then the quantum automorphism group of such 
    a triple is naturally defined. But one has to be careful as the isomorphism class 
    of $2$-graphs depend on a natural notion of equivalence between triples. 
    Then the quantum automorphism groups of such equivalent triples are isomorphic. 
    This allows us to call the quantum automorphism group of a 
    triple ``the'' quantum automorphism group of the corresponding $2$-graph. 
    
    Let us briefly mention the organization of the paper and some 
    of the main results obtained here. We start with a preliminary section 
    recalling the basics of higher rank (in particular rank $2$) graphs and 
    compact quantum groups. Then in the third section, 
    we define the notion of quantum automorphism group of a $2$-graph. 
    The underlying \(C^*\)-algebra of the quantum automorphism group is 
    described in terms of generators 
    and relations (see Equation \eqref{newrel} for a new kind of relations). 
    The next section (section 4) is devoted to some preliminary computations. 
    It is clear from the definition that in order to ensure that the quantum 
    automorphism group of a $2$-graph is genuine one has to at least start 
    with a triple which has two constituent graphs admitting genuine quantum 
    automorphisms. But that is naturally not sufficient though as seen in Example (b). 
    It is also clear from the definition that if one starts with two copies of the 
    same $1$-graph and the defining bijection as the trivial bijection, 
    the quantum automorphism group becomes the usual quantum automorphism 
    group of the constituent $1$-graph. Thus when one considers two copies of the same graph, 
    a non-trivial bijection becomes interesting. We exhibit one example where we 
    have two copies of the complete graph on $4$-vertices as the constituent graphs 
    of the triple, but there is a choice of a bijection so that the quantum automorphism 
    group of the corresponding triple is not only classical but also it is trivial. 
    In the last section, we define quantum isomorphism between a pair of $2$-graphs. 
    We show that the question whether a pair of $2$-graphs is quantum isomorphic 
    or not is intimately related to the quantum automorphism group of the 
    `disjoint union' of the pair (see Corollary \ref{qisocriterion}). 
    The results of the fifth section is a careful `extension' of the results already 
    explored in the literature (see for example \cite{Qinfo}) in the context of 
    usual $1$-graphs. Then we produce a `non-trivial' example of a pair 
    of $2$-graphs that are not quantum isomorphic (see Theorem \ref{main computation}). 
    This example of course is completely new. To produce such a pair we 
    consider a $2$-graph with `maximal' quantum symmetry and a $2$-graph 
    with `minimal' quantum symmetry. We end this paper by extending the free 
    wreath product result (\cite{bichon2004free}*{Theorem 4.2}) for $2$-graphs.

\section{Preliminaries}
\subsection{Higher rank graphs}
    Let us recall the definition of a higher rank graph briskly. 
    We shall be concerned about the case when $k=2$. 
    But for the sake of completeness we define a $k$-graph anyway. 
    For details, the reader is referred to \cite{NYJM}.
    
\begin{definition}[see~\cite{NYJM}*{Definition 1.1}]
  \label{defn_HR}
     A $k$-graph (rank $k$ graph or higher rank graph) $(\Lambda,d)$ 
     consists of a countable small category $\Lambda$ 
     (with target and source maps $t$ and $s$ respectively) 
     together with a functor $d:\Lambda\rightarrow\mathbb{N}^{k}$ 
     satisfying the factorization property: for every morphism 
     $\lambda\in \Lambda$ and $m,n\in\mathbb{N}^{k}$ with 
     $d(\lambda)=m+n,$ there are unique morphisms $\mu,\nu\in\Lambda$ 
     such that $\lambda=\mu\nu$ and $d(\mu)=m, \ d(\nu)=n.$ 
\end{definition}

\begin{remark}
    For $n\in\mathbb{N}^{k}$, one denotes the set of morphisms of degree $n$ 
    (i.e. $d^{-1}(n)$) by $\Lambda^{n}$. The unique factorization property enables 
    one to identify $\Lambda^{0}$ with ${\rm obj}(\Lambda)$. 
    We shall restrict ourselves to the case when $\Lambda^{0}$ is {\bf finite}.
\end{remark}

\begin{definition}[see~\cite{NYJM}*{Definition 1.9}]
     Let $f:\mathbb{N}^{l}\rightarrow\mathbb{N}^{k}$ be a monoid morphism. 
     If $(\Lambda,d)$ is a $k$-graph we may form the $l$-graph $f^{\ast}(\Lambda)$ 
     as follows:  $f^{\ast}(\Lambda)=\{(\lambda,n) : d(\lambda) = f(n)\}$ 
     with $d(\lambda,n)=n, 
     s(\lambda, n) = s(\lambda) \ {\rm and} \  t(\lambda, n) = t(\lambda)$.
\end{definition}

Now let us specialize to the case where $k=2$. 
We shall take an equivalent viewpoint of $2$-graphs which 
will be amenable to the study of quantum symmetry. This viewpoint is described 
in~\cite{NYJM}*{Section 6}. Here we shall give the bare bones. Roughly speaking, 
any $2$-graph can be constructed from two $1$-graphs with common set of 
vertices and commuting vertex matrices. Again we remind the reader that in our case, 
since ${\rm obj}(\Lambda)$ is finite, the number of common vertices is finite.

\medskip

\subsection*{\(2\)-graphs}
    Let $\mathcal{G}_{1}=(V,E_{1})$, $\mathcal{G}_{2}=(V,E_{2})$ 
    be two bidirected $1$-graphs with common set of vertices $V$ and 
    without multiple edges, loops or isolated vertices such that the 
    corresponding vertex matrices $\varepsilon_{1}, \varepsilon_{2}$ commute. 
    The commutation of $\varepsilon_{1}$ and $\varepsilon_{2}$ guarantees 
    the existence of a bijection $\theta:E_{1}\star E_{2}\raro E_{2}\star E_{1}$ where: 
    \begin{align*}
	E_{1}\star E_{2}
	    &=\left\{(e^{1}_{\alpha},e^{2}_{\beta})\in E_{1}\times E_{2}|s(e^{1}_{\alpha})
	      =t(e^{2}_{\beta})\right\}, \\ 
        E_{2}\star E_{1}
            &=\left\{(e^{2}_{\alpha},e^{1}_{\beta})\in E_{2}\times E_{1}|s(e^{2}_{\alpha})
              =t(e^{1}_{\beta})\right\}.
    \end{align*}
Given such a triple $(\mathcal{G}_{1},\mathcal{G}_{2},\theta)$, 
one can construct a $2$-graph $\Lambda$ in the following way: 
The set of objects ${\rm obj}(\Lambda)$ or $\Lambda^{0}$ is the set of common 
vertices of the graphs $\mathcal{G}_{1}, \mathcal{G}_{2}$. 
The morphisms of degree $(1,0)$ and $(0,1)$ are 
determined by the edge sets $E_{1},E_{2}$. 
The higher degree morphisms are built step by step using the bijection $\theta$. 
For details of the construction, we again refer the readers to~\cite{NYJM}*{Section 6}. 
Conversely, given a $2$-graph $\Lambda$, one can construct a triple. 
Roughly one takes the set $\Lambda^{(0,0)}$ to be the set of common vertices $V$. 
Then the morphisms $\Lambda^{(1,0)}$ and $\Lambda^{(0,1)}$ together with the 
target and source maps defines the constituent $1$-graphs $\mathcal{G}_{1}$ 
and $\mathcal{G}_{2}$. The bijection $\theta$ is then built out of the morphisms 
$\Lambda^{(1,1)}$ using the unique factorization property of the morphisms. 
A triple $(\mathcal{G}_{1},\mathcal{G}_{2},\theta)$ corresponding to 
a $2$-graph $\Lambda$ will be called the defining triple of $\Lambda$.

\begin{remark}
\((1)\) A bijection between the sets $E_{1}\star E_{2}$ and $E_{2}\star E_{1}$ 
is guaranteed precisely because the vertex matrices commute. Also, 
$\theta(e^{1}_{\alpha},e^{2}_{\beta})=(e^{2}_{\mu},e^{1}_{\nu})$ implies
\begin{equation*}
  \begin{matrix}
	 s(e^{1}_{\alpha})=t(e^{2}_{\beta}); & s(e^{2}_{\mu})=t(e^{1}_{\nu});\\
	 t(e^{1}_{\alpha})=t(e^{2}_{\mu}); & s(e^{2}_{\beta})=s(e^{1}_{\nu}).
   \end{matrix}
\end{equation*}
We shall refer to a $2$-graph by the triple $(\mathcal{G}_{1},\mathcal{G}_{2},\theta)$. 
Observe that given the bijection $\theta$, one has a natural vector space isomorphism 
$\theta_{\ast}:C(E_{1}\star E_{2})\raro C(E_{2}\star E_{1})$ given by 
$\theta_{\ast}(f)(e^{2}_{\mu},e^{1}_{\nu}):=
f(\theta^{-1}(e^{2}_{\mu},e^{1}_{\nu}))$ for $f\in C(E_{1}\star E_{2})$.

\medskip

\noindent \((2)\) We also mention that we restrict our attention to the $2$-graphs 
where the individual $1$-graphs are bidirected without 
{\bf multiple edges, loops or isolated vertices} for technical reasons. We assume that the $1$-graphs are without multiple edges and loops to consider ${\rm Aut}^{+}(\mathcal{G}_{1}), {\rm Aut}^{+}(\mathcal{G}_{2}) $ [See  Lemma \ref{Fulton_rel}]. The assumption that the graphs are bidirected and without isolated vertices will be used in the proof of Theorem~\ref{main theorem}    [See Section~\ref{main result}] and Lemma~\ref{well defined} [See Section \ref{wreath product}]. 
\end{remark}

\begin{remark}
  \label{pullback}
    A case of particular interest is when we have two copies of a $1$-graph 
    say $\mathcal{G}=(V,E)$ and the identity bijection $\theta$ from $E\star E$ to $E\star E$. 
    In that case the resulting $2$-graph is isomorphic to the $2$-graph 
    $f^{\ast}(\mathcal{G})$ where $f:\mathbb{N}^{2}\rightarrow\mathbb{N}$ 
    is the monoid morphism $f(m,n)=m+n$.
\end{remark}

Now let $(\mathcal{G}_{1}, \mathcal{G}_{2},\theta)$ and 
$(\mathcal{G}_{1}^{\prime},\mathcal{G}^{\prime}_{2},\theta^{\prime})$ be two triples. 
We denote the common vertex set of the triples by $V, V^{\prime}$ respectively; 
the vertex matrices of $\mathcal{G}_{l}$ by $\varepsilon_{l}$ and vertex matrices of 
$\mathcal{G}^{\prime}_{l}$ by $\varepsilon_{l}^{\prime}$ for $l=1,2$ respectively; 
the edge sets of $\mathcal{G}_{l}$ and $\mathcal{G}_{l}^{\prime}$ by $E_{l}$ and 
$E_{l}^{\prime}$ for $l=1,2$ respectively.

\begin{definition}
  \label{equivalent_triple}
    Two triples $(\mathcal{G}_{1}, \mathcal{G}_{2}, \theta)$ and 
    $(\mathcal{G}_{1}^{\prime},\mathcal{G}_{2}^{\prime},\theta^{\prime})$ 
    as above are said to be equivalent if the following holds:
    \begin{enumerate}[(i)]
        \item There is a bijection $T:V\rightarrow V^{\prime}$. 
        		Consequently, it induces a linear map $T_{\ast}:C(V^{\prime})\rightarrow C(V)$.
        \item If we denote the permutation matrix of the linear map 
        		$T_{\ast}$ again by $T$ with respect to standard bases of
		 $C(V)$ and $C(V^{\prime})$, then 
		 $T\varepsilon_{l}^{\prime}=\varepsilon_{l}T$ for $l=1,2$.
    	\item Consider the map 
		$T^{\ot 4}_{\ast}\equiv T_{\ast}\ot T_{\ast}\ot T_{\ast}\ot T_{\ast}: 
		C(V^{\prime}\times V^{\prime}\times V^{\prime}\times V^{\prime})
		\rightarrow 
		C(V\times V\times V\times V)$. 
		Thanks to condition (ii), $T^{\ot 4}_{\ast}$ is also a bijection between 
		both $C(E_{1}^{\prime}\star E_{2}^{\prime})$, $C(E_{1}\star E_{2})$ 
		and $C(E_{2}^{\prime}\star E_{1}^{\prime})$, $C(E_{2}\star E_{1})$. 
		We have the following commutative diagram:
    		\begin{center}
		    \begin{tikzcd}
			C(E_{1}^{\prime}\star E_{2}^{\prime}) \arrow{r}{T^{\ot 4}_{\ast}} 
			\arrow[swap]{d}{\theta_{\ast}^{\prime}} & C(E_{1}\star E_{2}) 
			\arrow{d}{\theta_{\ast}} \\
			C(E_{2}^{\prime}\star E_{1}^{\prime}) \arrow{r}{T^{\ot 4}_{\ast}} & C(E_{2}\star E_{1})
		    \end{tikzcd}
		\end{center}
    \end{enumerate}
\end{definition}

\begin{remark}
We abbreviate the last condition of the above definition as $\theta$ -compatibility. 
The isomorphism classes of $2$-graphs depend only on the equivalence classes 
of triples as defined above. We have included a brief sketch of the proof. 
We believe that it must be known. But since we could not find an explicit 
mention of this result in the literature, we have decided to keep the result 
in the Appendix (see Theorem \ref{appendix}).
\end{remark}

\subsection{Compact quantum groups and their actions} 
The term quantum group is quite general. We make it clear that we work 
with compact quantum groups (CQG in short) ala S.L. Woronowicz (\cite{Woro}). 
The literature on CQG and their actions on $C^{\ast}$-algebras are aplenty by now. 
The reader can refer to the articles \cites{DeCom,Soltan} for details on the 
actions of CQG's on $C^{\ast}$-algebras. Here we briefly recall what is relevant for us. 
The following are mainly taken from \cite{edin}.

\begin{definition}
A compact quantum group $\mathbb{G}$ is a pair $(C(\mathbb{G}), \Delta_{\mathbb{G}})$ 
such that $C(\mathbb{G})$ is a unital $C^{\ast}$-algebra and 
$\Delta_{\mathbb{G}}:C(\mathbb{G})\rightarrow C(\mathbb{G})\otimes C(\mathbb{G})$ 
is a unital $C^{\ast}$-
homomorphism satisfying
  \begin{enumerate}[(i)]
    \item $({\rm id}\otimes\Delta_{\mathbb{G}})\circ\Delta_{\mathbb{G}}
    		= (\Delta_{\mathbb{G}}\otimes {\rm id})\circ\Delta_{\mathbb{G}};$
    \item ${\rm Span}\{\Delta_{\mathbb{G}}(C(\mathbb{G}))(1\otimes C(\mathbb{G}))\}$ 
    		and ${\rm Span}\{\Delta_{\mathbb{G}}(C(\mathbb{G}))(C(\mathbb{G})\otimes 1)\}$ 
    		are dense in $C(\mathbb{G})\otimes C(\mathbb{G})$.
  \end{enumerate}
\end{definition}

\begin{remark} 
Given a CQG $\mathbb{G}$, there is a canonical dense Hopf $\ast$-algebra 
$C(\mathbb{G})_{0}$ in $C(\mathbb{G})$ on which
an antipode $\kappa$ and counit $\epsilon$ are defined. 
Given two CQG’s $\mathbb{G}_{1}$ and $\mathbb{G}_{2}$, a CQG morphism
between them is a unital $C^{\ast}$-homomorphism 
$\pi:C(\mathbb{G}_1)\rightarrow C(\mathbb{G}_2)$ such that 
$(\pi\otimes\pi)\circ\Delta_{\mathbb{G}_{1}} = \Delta_{\mathbb{G}_{2}}\circ\pi.$
\end{remark}

\begin{definition}
\label{faithfull action}
Given a (unital) $C^{\ast}$-algebra $\clc$, a CQG $(C(\mathbb{G}),\Delta_{\mathbb{G}})$ 
is said to act faithfully on $\clc$ if there is a unital $C^{\ast}$-homomorphism 
$\alpha:\clc\raro\clc\ot C(\mathbb{G})$ satisfying
  \begin{enumerate}[(i)]
    \item $(\alpha\ot {\rm id})\circ\alpha=({\rm id}\ot \Delta_{\mathbb{G}})\circ\alpha$;
    \item {\rm Span}$\{\alpha(\clc)(1\ot \Delta_{\mathbb{G}})\}$ 
    	is dense in $\clc\ot C(\mathbb{G})$;
    \item The $\ast$-algebra generated by the set  
    	$\{(\omega\ot{\rm id})\circ\alpha(\clc): \omega\in\clc^{\ast}\}$ 
	is norm-dense in $C(\mathbb{G})$.
  \end{enumerate}
\end{definition}

\begin{definition}[\cite{Bichon}*{Definition 2.1}]
    Given a unital $C^{\ast}$-algebra $\clc$, the quantum automorphism group of 
    $\clc$ is a CQG $(C(\mathbb{G}),\Delta_{\mathbb{G}})$ acting faithfully on 
    $\clc$ satisfying the following universal property:
    
    \begin{minipage}[h]{0.9\textwidth}
	 If $\mathbb{B}$ is any CQG acting faithfully on $\clc$ with the action $\beta$, 
	 there is a surjective CQG morphism $\pi:C(\mathbb{G})\raro C(\mathbb{B})$ 
	 such that $({\rm id}\ot \pi)\circ\alpha=\beta$. 
    \end{minipage}
\end{definition}

\begin{remark}
In general, the universal object may fail to exist in the above category. 
To ensure its existence, one typically assumes that the action preserves some 
fixed state on the $C^{\ast}$-algebra. However, for certain finite dimensional 
$C^{\ast}$-algebras, such a state preserving condition is automatic. 
We will not go into further details here as we will not be using it in this paper. 
For more details, the reader may consult \cite{Wang}.
\end{remark}

\begin{example}
  \label{S}
    If we take the space of $n$ points $X_{n}$ then the quantum automorphism 
    group of the $C^{\ast}$-algebra $C(X_{n})$ is given by the CQG 
    (denoted by $S_{n}^{+}$) whose underlying $C^{\ast}$-algebra is the 
    universal $C^{\ast}$ algebra generated by $\{q_{ij}\}_{i,j=1,\ldots,n}$ 
    satisfying the following relations (see~\cite{Wang}*{Theorem 3.1}):
	\begin{displaymath}
		q_{ij}^{2}=q_{ij},
		\quad 
		q_{ij}^{\ast}=q_{ij},
		\quad 
		\sum_{k=1}^{n}q_{ik}=\sum_{k=1}^{n}q_{ki}=1,
		\quad \text{ for } i,j=1,\ldots,n.
	\end{displaymath}
    Any $n\times n$ matrix $((q_{ij}))_{i,j=1,\ldots,n}$ with $q_{ij}$'s in a $C^{\ast}$-algebra 
    satisfying the above relations is called a {\bf magic} unitary. 
    The coproduct on the generators is given by $\Delta(q_{ij})=\sum_{k=1}^{n}q_{ik}\ot q_{kj}$. 
    It is worth mentioning that the CQG $S_{n}^{+}$ is a particular type of CQG 
    known as the compact matrix pseudogroup (see \cite{Woro}). 
    The $\ast$-algebra generated by $\{q_{ij}\}_{i,j=1,\ldots,n}$ is the canonical dense 
    Hopf $\ast$-algebra and the antipode \(\kappa\) 
    is given on the generators by $\kappa(q_{ij})=q_{ji}$.
\end{example}

\subsection{Quantum automorphism group of finite graphs}
Recall the definition of a finite, directed graph 
$\mathcal{G}=(V=\{v_{1},\ldots,v_{n}\}, E= \{e_{1},\ldots, e_{m}\})$ 
without multiple edges and the CQG $S_{n}^{+}$. 
We shall call a finite directed graph - a $1$-graph, 
to be consistent with higher rank graph terminology from now on.

\begin{definition}
  \label{qsymban}
The quantum automorphism group of a graph $\mathcal{G}$ 
without multiple edges (to be denoted by ${\rm Aut}^{+}(\mathcal{G})$) 
is defined to be the quotient $S^{+}_{n}/\langle A\varepsilon-\varepsilon A\rangle$, 
where $A=((q_{ij}))_{i,j=1,\ldots,n}$, and $\varepsilon$ is the vertex matrix of $\mathcal{G}$. 
The coproduct on the generators is  given by $\Delta(q_{ij})=\sum_{k=1}^{n}q_{ik}\ot q_{kj}$.
\end{definition}

For the classical automorphism group ${\rm Aut}(\mathcal{G})$, 
the commutative CQG $C({\rm Aut}(\mathcal{G}))$ is generated by 
$q_{ij}$'s where $q_{ij}$ is a function on $S_{n}$ taking value $1$ 
on the permutation which sends $i$-th vertex to $j$-th vertex and takes 
the value zero on other elements of the group. It is a quantum subgroup 
of ${\rm Aut}^{+}(\mathcal{G})$. The surjective CQG morphism 
$\pi:C({\rm Aut}^{+}(\mathcal{G}))\raro C({\rm Aut}(\mathcal{G}))$ 
sends the generators to generators. We recall the following

\begin{lemma}[\cite{Fulton}*{Lemma 3.1.1}]
  \label{Fulton_rel}
   The underlying $C^{\ast}$-algebra of the quantum automorphism group 
   ${\rm Aut}^{+}(\mathcal{G})$ of a finite graph $\mathcal{G}$ with $m$ edges 
   and $n$ vertices (without multiple edges) is the universal $C^{\ast}$-algebra 
   generated by $\{q_{ij}\}_{i,j=1,\ldots,n}$ satisfying the following relations:
   \begin{enumerate}
	\item \( q_{ij}^{\ast}=q_{ij},
		\quad 
		q_{ij}q_{ik}=\delta_{jk}q_{ij},
		\quad 
		q_{ji}q_{ki}=\delta_{jk}q_{ji}, 
		\text{ for } 1\leq i,j,k\leq n;\)
	\item \( \sum_{l=1}^{n}q_{il}=\sum_{l=1}^{n}q_{li}=1, 
		\text{ for } 1\leq i\leq n; \label{1}\)
	\item \(q_{s(e_{j})i}q_{t(e_{j})k}
		=q_{t(e_{j})k}q_{s(e_{j})i}
		=q_{is(e_{j})}q_{kt(e_{j})}
		=q_{kt(e_{j})}q_{is(e_{j})}
		=0, \) 
		whenever \( e_{j}\in E\) and \((i,k)\not\in E. \label{2}\)
\end{enumerate}
The coproduct on the generators is given by 
\(\Delta(q_{ij})=\sum_{k=1}^{n}q_{ik}\ot q_{kj}\). 
The \(C^{\ast}\)-action on the graph is given by 
\(\alpha(\delta_{i})=\sum_{j}\delta_{j}\ot q_{ji}\), 
where \(\delta_{i}\) is the function which takes 
value \(1\) on \(i\)-th vertex and zero elsewhere.
\end{lemma}

\section{Main Results} \label{main result}
 To define the quantum automorphism group of a \(2\)-graph, 
 we shall define the notion of the quantum automorphism group of a triple 
 \((\mathcal{G}_{1},\mathcal{G}_{2},\theta)\) first. 
 Then we shall show that the quantum automorphism groups of two 
 such equivalent triples are isomorphic. Then as the isomorphism class 
 of a $2$-graph depends on the equivalence class of its defining triple, 
 one can define the quantum automorphism group of a $2$-graph to be 
 the quantum automorphism group of its defining triple unambiguously. 
 To that end, let $(\mathcal{G}_{1},\mathcal{G}_{2},\theta)$ be a triple 
 where we remind the readers that $\mathcal{G}_{l}=(V,E_{l})$ are $1$-graphs 
 (bidirected and without any {\bf multiple edge, loops, or isolated vertices}) with the common 
 set of finitely many vertices $V$ and edge sets $E_{l}$'s for $l=1,2$ respectively; 
 Let the vertex matrices of $\mathcal{G}_{l}$ be $\varepsilon_{l}$ for $l=1,2$ 
 respectively such that $\varepsilon_{1}\varepsilon_{2}=\varepsilon_{2}\varepsilon_{1}$. 
 Given a CQG $\mathbb{G}$ having an action $\alpha:C(V)\raro C(V)\ot C(\mathbb{G})$ 
 given by $\alpha(\delta_{v_{i}})=\sum_{j}\delta_{v_{j}}\ot q_{ji}$ for $q_{ji}\in C(\mathbb{G})$ 
 such that $((q_{ij}))$ commutes with $\varepsilon_{1}$ and $\varepsilon_{2}$, 
 there are corepresentations $\beta_{l}$ on $C(E_{l})$ given by
\begin{displaymath}
	\beta_{l}(\delta_{e^{l}_{\alpha}})
	=\sum_{\beta}\delta_{e^{l}_{\beta}}
	\ot q_{t(e^{l}_{\beta})t(e^{l}_{\alpha})}
	q_{s(e^{l}_{\beta})s(e^{l}_{\alpha})},
\end{displaymath}
where \(e^{l}_{\alpha}\)'s and \(e^{l}_{\beta}\)'s are edges.
Then one can consider tensor product corepresentations 
$\beta_{1}\ot\beta_{2}$ and $\beta_{2}\ot\beta_{1}$ on 
$C(E_{1}\times E_{2})$ and $C(E_{2}\times E_{1})$ respectively. 
Now note that $C(E_{1}\star E_{2})$ is a linear subspace 
of $C(E_{1}\times E_{2})$ with a basis for $C(E_{1}\star E_{2})$ 
given by $\{\delta_{e^{1}_{\alpha}}\ot\delta_{e^{2}_{\beta}}:s(e^{1}_{\alpha})=t(e^{2}_{\beta})\}$. 
Similarly $C(E_{2}\star E_{1})$ is a linear subspace of $C(E_{2}\times E_{1})$.

\begin{lemma}
\label{invariant}
The corepresentation $\beta_{1}\ot\beta_{2}$ leaves the subspace $C(E_{1}\star E_{2})$ invariant.
\end{lemma}

\begin{proof}
	By definition of $\beta_{1}\ot\beta_{2}$, we have the following:
 	\begin{equation*}	
	(\beta_{1}\ot\beta_{2})(\delta_{e^{1}_{\alpha}}\ot\delta_{e^{2}_{\beta}})
	=\sum \delta_{e^{1}_{\mu}}\ot\delta_{e^{2}_{\nu}}
	   \ot q_{t(e^{1}_{\mu})t(e^{1}_{\alpha})}
	   q_{s(e^{1}_{\mu})s(e^{1}_{\alpha})}
	   q_{t(e^{2}_{\nu})t(e^{2}_{\beta})}
	   q_{s(e^{2}_{\nu})s(e^{2}_{\beta})}.
	\end{equation*}
Now suppose $s(e^{1}_{\alpha})=t(e^{2}_{\beta})$ so that 
$\delta_{e^{1}_{\alpha}}\ot\delta_{e^{2}_{\beta}}\in C(E_{1}\star E_{2})$. 
Then for $(e^{1}_{\mu},e^{2}_{\nu})$ such that $s(e^{1}_{\mu})\neq t(e^{2}_{\nu})$, 
we have 
$q_{t(e^{1}_{\mu})t(e^{1}_{\alpha})}
q_{s(e^{1}_{\mu})s(e^{1}_{\alpha})}
q_{t(e^{2}_{\nu})t(e^{2}_{\beta})}
q_{s(e^{2}_{\nu})s(e^{2}_{\beta})}=0$ which proves the lemma.
\end{proof} 

\begin{remark}
	(i) By the same reasoning one can show that $\beta_{2}\ot\beta_{1}$ 
	leaves $C(E_{2}\star E_{1})$ invariant. We denote the restriction of 
	$\beta_{1}\ot\beta_{2}$ on $C(E_{1}\star E_{2})$ by $\beta$ and 
	the restriction of $\beta_{2}\ot\beta_{1}$ on $C(E_{2}\star E_{1})$ 
	by $\beta^{{\rm op}}$.
	
        \noindent(ii) The corepresentations $\beta$, $\beta^{\rm op}$ are nothing 
        but the restrictions of the tensor product corepresentation 
        $\alpha^{\ot 4}:=\alpha\ot\alpha\ot\alpha\ot\alpha$ of 
        $\mathbb{G}$ on $C(V\times V\times V\times V)$ to the 
        subspaces $C(E_{1}\star E_{2})$ and $C(E_{2}\star E_{1})$ respectively.
\end{remark}
	
 Now recall the vector space isomorphism \(\theta_{\ast}\) corresponding to the bijection \(\theta\).
 
\begin{definition}
	\label{qaut} 
	A CQG $\mathbb{G}$ is said to act on a triple $(\mathcal{G}_{1},\mathcal{G}_{2},\theta)$, 
	if there is an action $\alpha:C(V)\raro C(V)\ot C(\mathbb{G})$ given by 
	$\alpha(\delta_{v_{i}})=\sum_j\delta_{v_{j}}\ot q_{ji}$ such that
	\begin{enumerate}[(i)]
	  \item The matrix $((q_{ij}))$ commutes with the vertex matrices of 
	          $\mathcal{G}_{1}$ and $\mathcal{G}_{2}$ so that we have the corresponding 
	          linear maps $\beta$ and $\beta^{{\rm op}}$ on $C(E_{1}\star E_{2})$ 
	          and $C(E_{2}\star E_{1})$ respectively.
          \item The following diagram commutes:
		\begin{center}
		   \begin{tikzcd}
			C(E_{1}\star E_{2}) \arrow{r}{\beta} \arrow[swap]{d}{\theta_{\ast}} 
			& C(E_{1}\star E_{2})\ot C(\mathbb{G}) \arrow{d}{\theta_{\ast}\ot {\rm id}} \\
			C(E_{2}\star E_{1}) \arrow{r}{\beta^{\rm op}} & C(E_{2}\star E_{1})\ot C(\mathbb{G})
		  \end{tikzcd}
		\end{center}
	\end{enumerate}
         The action is said to be faithful if $\alpha: C(V)\raro C(V)\ot C(\mathbb{G})$ is faithful in the sense of Definition \ref{faithfull action}.
\end{definition}

\subsection*{Generators and relations} 
Let $\mathbb{G}$ be a CQG acting on a triple $(\mathcal{G}_{1},\mathcal{G}_{2},\theta)$ 
so that (assume there are $n$-vertices) $\alpha:C(V)\raro C(V)\ot C(\mathbb{G})$ 
is a faithful action given by $\alpha(\delta_{v_{i}})=\sum_j\delta_{v_{j}}\ot q_{ji}$. 
Then $C(\mathbb{G})$ is generated by $\{q_{ij}\}_{i,j=1,\ldots,n}$. 
Now $q_{ij}$'s satisfy all the relations of ${\rm Aut}^{+}(\mathcal{G}_{l})$ for $l=1,2$ 
as described in Theorem \ref{Fulton_rel}. Let us figure out the relations imposed by 
the commutative diagram in the Definition \ref{qaut}. We start with the easy observation 
that if $\theta(e^{1}_{\alpha},e^{2}_{\beta})=(e^{2}_{\mu},e^{1}_{\nu})$, then 
$\theta_{\ast}(\delta_{e^{1}_{\alpha}}\ot \delta_{e^{2}_{\beta}})
=\delta_{e^{2}_{\mu}}\ot \delta_{e^{1}_{\nu}}$. 
Let $(e^{1}_{\alpha},e^{2}_{\beta})\in E_{1}\star E_{2}$ such that 
$\theta(e^{1}_{\alpha},e^{2}_{\beta})=(e^{2}_{\mu}, e^{1}_{\nu})$ 
for some $(e^{2}_{\mu},e^{1}_{\nu})\in E_{2}\star E_{1}$. Then \\
\begin{align}
	&\beta^{\rm op}\circ\theta_{\ast}(\delta_{e^{1}_{\alpha}}\ot\delta_{e^{2}_{\beta}})\nonumber\\
    	&=\beta^{\rm op}(\delta_{e^{2}_{\mu}}\ot\delta_{e^{1}_{\nu}})\nonumber\\
	&= \sum_{(e^{2}_{\mu^{\prime}},e^{1}_{\nu^{\prime}})\in E_{2}\star E_{1}}\delta_{e^{2}_{\mu^{\prime}}}
		\ot\delta_{e^{1}_{\nu^{\prime}}}
		\ot q_{t(e^{2}_{\mu^{\prime}})t(e^{2}_{\mu})}
		q_{s(e^{2}_{\mu^{\prime}})s(e^{2}_{\mu})}
		q_{t(e^{1}_{\nu^{\prime}})t(e^{1}_{\nu})}
		q_{s(e^{1}_{\nu^{\prime}})s(e^{1}_{\nu})}\nonumber\\
	&= \sum_{(e^{2}_{\mu^{\prime}},e^{1}_{\nu^{\prime}})\in E_{2}\star E_{1}}\delta_{e^{2}_{\mu^{\prime}}}
		\ot\delta_{e^{1}_{\nu^{\prime}}}
		\ot q_{t(e^{2}_{\mu^{\prime}})t(e^{2}_{\mu})}
		q_{s(e^{2}_{\mu^{\prime}})s(e^{2}_{\mu})}
		q_{s(e^{1}_{\nu^{\prime}})s(e^{1}_{\nu})}.
\end{align} 
On the other hand 
\begin{align*}
	&(\theta_{\ast}\ot{\rm id})\circ\beta(\delta_{e^{1}_{\alpha}}\ot\delta_{e^{2}_{\beta}})\\
	&= (\theta_{\ast}\ot{\rm id})
		\left(
		\sum_{(e^{1}_{\alpha^{\prime}},e^{2}_{\beta^{\prime}})\in E_{1}\star E_{2}} 
		\delta_{e^{1}_{\alpha^{\prime}}}
		\ot\delta_{e^{2}_{\beta^{\prime}}}
		\ot q_{t(e^{1}_{\alpha^{\prime}})t(e^{1}_{\alpha})}
		q_{s(e^{1}_{\alpha^{\prime}})s(e^{1}_{\alpha})}
		q_{t(e^{2}_{\beta^{\prime}})t(e^{2}_{\beta})}
		q_{s(e^{2}_{\beta^{\prime}})s(e^{2}_{\beta})}\right)
\end{align*}
Now let 
$\theta(e^{1}_{\alpha^{\prime}},e^{2}_{\beta^{\prime}})
=(e^{2}_{\mu^{\prime}},e^{1}_{\nu^{\prime}})$ 
and observe that as $\theta$ is a bijection, if 
$(e^{1}_{\alpha^{\prime}},e^{2}_{\beta^{\prime}})$ 
varies over all 
$E_{1}\star E_{2}$, $(e^{2}_{\mu^{\prime}},e^{1}_{\nu^{\prime}})$ 
varies over all of $E_{2}\star E_{1}$ so that 
\begin{align}
	&(\theta_{\ast}\ot{\rm id})
	\circ
	\beta(\delta_{e^{1}_{\alpha}}\ot\delta_{e^{2}_{\beta}}) \nonumber\\
	&= \sum_{(e^{2}_{\mu^{\prime}},e^{1}_{\nu^{\prime}})\in E_{2}\star E_{1}} 
	\delta_{e^{2}_{\mu^{\prime}}}
	\ot
	\delta_{e^{1}_{\nu^{\prime}}}
	\ot 
	q_{t(e^{1}_{\alpha^{\prime}})t(e^{1}_{\alpha})}
	q_{s(e^{1}_{\alpha^{\prime}})s(e^{1}_{\alpha})}
	q_{s(e^{2}_{\beta^{\prime}})s(e^{2}_{\beta})}.
\end{align} 
Comparing the coefficients, we get for all 
$(e^{1}_{\alpha},e^{2}_{\beta})\in E_{1}\star E_{2}$ 
and all 
$(e^{2}_{\mu^{\prime}},e^{1}_{\nu^{\prime}})\in E_{2}\star E_{1}$ 
with 
$\theta(e^{1}_{\alpha},e^{2}_{\beta})=(e^{2}_{\mu},e^{1}_{\nu})$ 
and 
$\theta(e^{1}_{\alpha^{\prime}},e^{2}_{\beta^{\prime}})
=(e^{2}_{\mu^{\prime}},e^{1}_{\nu^{\prime}})$,
\begin{equation}
	\label{newrel}
	q_{t(e^{2}_{\mu^{\prime}})t(e^{2}_{\mu})}
	q_{s(e^{2}_{\mu^{\prime}})s(e^{2}_{\mu})}
	q_{s(e^{1}_{\nu^{\prime}})s(e^{1}_{\nu})}
	=
	q_{t(e^{1}_{\alpha^{\prime}})t(e^{1}_{\alpha})}
	q_{s(e^{1}_{\alpha^{\prime}})s(e^{1}_{\alpha})}
	q_{s(e^{2}_{\beta^{\prime}})s(e^{2}_{\beta})}.
\end{equation}

\subsection*{Observation}
 Recall that 
 $\theta(e^{1}_{\alpha},e^{2}_{\beta})=(e^{2}_{\mu},e^{1}_{\nu})$ 
 and 
 $\theta(e^{1}_{\alpha^{\prime}},e^{2}_{\beta^{\prime}})
 =(e^{2}_{\mu^{\prime}},e^{1}_{\nu^{\prime}})$ 
 together imply 
 $t(e^{2}_{\mu^{\prime}})=t(e^{1}_{\alpha^{\prime}})$, 
 $t(e^{2}_{\mu})=t(e^{1}_{\alpha})$, 
 $s(e^{1}_{\nu^{\prime}})=s(e^{2}_{\beta^{\prime}})$ 
 and $s(e^{1}_{\nu})=s(e^{2}_{\beta})$. 
 Therefore the relations \eqref{newrel} have a few equivalent variations as 
 $q_{t(e^{2}_{\mu^{\prime}})t(e^{2}_{\mu})}
 =q_{t(e^{1}_{\alpha^{\prime}})t(e^{1}_{\alpha})}$ 
 and 
 $q_{s(e^{1}_{\nu^{\prime}})s(e^{1}_{\nu})}
 =q_{s(e^{2}_{\beta^{\prime}})s(e^{2}_{\beta})}$.
 
\begin{definition}
  \label{triple_def}
	The quantum automorphism group of a triple 
	$(\mathcal{G}_{1},\mathcal{G}_{2},\theta)$ 
	(to be denoted by ${\rm Aut}^{+}(\mathcal{G}_{1},\mathcal{G}_{2},\theta)$) 
	is defined to be the universal CQG in the category of CGQ's acting \emph{faithfully} on the 
	triple in the sense of Definition~\ref{qaut}.
\end{definition}

\begin{theorem}
  \label{main theorem}
    Let  $(\mathcal{G}_{1},\mathcal{G}_{2},\theta)$ 
    be a triple such that $\mathcal{G}_{1},\mathcal{G}_{2}$ 
    have a common set of vertices $V$ with $n$-points and 
    both the $1$-graphs are bidirected without any multiple edge, loops or isolated vertices. 
    Then ${\rm Aut}^{+}(\mathcal{G}_{1},\mathcal{G}_{2},\theta)$ exists. 
    The underlying 
    $C^{\ast}$-algebra of ${\rm Aut}^{+}(\mathcal{G}_{1},\mathcal{G}_{2},\theta)$ 
    is the universal $C^{\ast}$-algebra generated by $n^{2}$-elements $q_{ij}$'s 
    such that
	\begin{enumerate}[(i)]
	\item The elements $q_{ij}$'s satisfy the relations of 
		${\rm Aut}^{+}(\mathcal{G}_{l})$ for $l=1,2$.
	\item The elements $q_{ij}$'s satisfy the relations \eqref{newrel}.
		The coproduct is given on the generators by 
		$\Delta(q_{ij})=\sum_{k=1}^{n}q_{ik}\ot q_{kj}$. 
		If the usual basis for $C(V)$ is given by $\{\delta_{i}\}_{i=1}^{n}$, 
		the coaction on $C(V)$ is given by 
		$\alpha(\delta_{i})=\sum_{j=1}^{n}\delta_{j}\ot q_{ji}$ for $i=1,\ldots,n$.
	\end{enumerate}
\end{theorem}

\begin{proof}
	Note that it suffices to show that the ideal generated by the 
	relations~\eqref{newrel} is a Woronowicz ideal i.e. for all 
	$(e^{1}_{\alpha},e^{2}_{\beta})\in E_{1}\star E_{2}$ 
	and all 
	$(e^{2}_{\mu^{\prime}},e^{1}_{\nu^{\prime}})\in E_{2}\star E_{1}$ 
	with 
	$\theta(e^{1}_{\alpha},e^{2}_{\beta})=(e^{2}_{\mu},e^{1}_{\nu})$ 
	and 
	$\theta(e^{1}_{\alpha^{\prime}},e^{2}_{\beta^{\prime}})
	=(e^{2}_{\mu^{\prime}},e^{1}_{\nu^{\prime}})$,
	\begin{equation*}
	    \Delta
	    \left(
	    q_{t(e^{2}_{\mu^{\prime}})t(e^{2}_{\mu})}
	    q_{s(e^{2}_{\mu^{\prime}})s(e^{2}_{\mu})}
	    q_{s(e^{1}_{\nu^{\prime}})s(e^{1}_{\nu})}
	    \right)
	    =
	    \Delta
	    \left(
	    q_{t(e^{1}_{\alpha^{\prime}})t(e^{1}_{\alpha})}
	    q_{s(e^{1}_{\alpha^{\prime}})s(e^{1}_{\alpha})}
	    q_{s(e^{2}_{\beta^{\prime}})s(e^{2}_{\beta})}
	    \right).
    \end{equation*}
    The left hand side is equal to 
     \begin{displaymath}
	\sum_{i,j,k\in V}q_{t(e^{2}_{\mu^{\prime}})i}
	q_{s(e^{2}_{\mu^{\prime}})j}
	q_{s(e^{1}_{\nu^{\prime}})k}
	\ot 
	q_{it(e^{2}_{\mu})}
	q_{js(e^{2}_{\mu})}
	q_{ks(e^{1}_{\nu})}.
    \end{displaymath}
    Now for $(j,i)\notin E_{2}$, 
    $q_{t(e^{2}_{\mu^{\prime}})i}q_{s(e^{2}_{\mu^{\prime}})j}=0$. 
    As none of the graphs $\mathcal{G}_{l}$ admits multiple edges, 
    the last summation reduces to
    \begin{displaymath}
	\sum_{f^{2}_{\mu^{\prime\prime}}\in E_{2},k\in V} 
	q_{t(e^{2}_{\mu^{\prime}})t(f^{2}_{\mu^{\prime\prime}})}
	q_{s(e^{2}_{\mu^{\prime}})s(f^{2}_{\mu^{\prime\prime}})}
	q_{s(e^{1}_{\nu^{\prime}})k}
	\ot 
	q_{t(f^{2}_{\mu^{\prime\prime}})t(e^{2}_{\mu})}
	q_{s(f^{2}_{\mu^{\prime\prime}})s(e^{2}_{\mu})}
	q_{ks(e^{1}_{\nu})}.
    \end{displaymath}
    Since $\mathcal{G}_{1}$ is bidirected and without isolated vertices, 
    there is an edge $f^{1}_{\alpha^{\prime\prime}}\in E_{1}$ 
    such that $s(f^{2}_{\mu^{\prime\prime}})=t(f^{1}_{\alpha^{\prime\prime}})$. 
    In particular this implies that 
    $(f^{2}_{\mu^{\prime\prime}},f^{1}_{\alpha^{\prime\prime}})\in E_{2}\star E_{1}$. 
    Replacing $s(e^{2}_{\mu^{\prime}})$ by $t(e^{1}_{\nu^{\prime}})$ 
    and using the fact that 
    $q_{t(e^{1}_{\nu^{\prime}})t(f^{1}_{\alpha^{\prime\prime}})}
    q_{s(e^{1}_{\nu^{\prime}})k}=0$ for $k\neq s(f^{1}_{\alpha^{\prime\prime}})$, 
    the last summation reduces to
    \begin{equation*}
	\qquad
	\sum_{\mathclap{(f^{2}_{\mu^{\prime\prime}},f^{1}_{\alpha^{\prime\prime}})\in E_{2}\star E_{1}}} 
	q_{t(e^{2}_{\mu^{\prime}})t(f^{2}_{\mu^{\prime\prime}})}
	q_{t(e^{1}_{\nu^{\prime}})t(f^{1}_{\alpha^{\prime\prime}})}
	q_{s(e^{1}_{\nu^{\prime}})s(f^{1}_{\alpha^{\prime\prime}})}
	\ot
	q_{t(f^{2}_{\mu^{\prime\prime}})t(e^{2}_{\mu})}
	q_{t(f^{1}_{\alpha^{\prime\prime}})s(e^{2}_{\mu})}
	q_{s(f^{1}_{\alpha^{\prime\prime}})s(e^{1}_{\nu})}
    \end{equation*}
    which is equal to
    \begin{equation*}
      \qquad
      \sum_{\mathclap{(f^{2}_{\mu^{\prime\prime}},f^{1}_{\alpha^{\prime\prime}})\in E_{2}\star E_{1}}} 
      q_{t(e^{2}_{\mu^{\prime}})t(f^{2}_{\mu^{\prime\prime}})}
      q_{s(e^{2}_{\mu^{\prime}})s(f^{2}_{\mu^{\prime\prime}})}
      q_{s(e^{1}_{\nu^{\prime}})s(f^{1}_{\alpha^{\prime\prime}})}
      \ot 
      q_{t(f^{2}_{\mu^{\prime\prime}})t(e^{2}_{\mu})}
      q_{s(f^{2}_{\mu^{\prime\prime}})s(e^{2}_{\mu})}
      q_{s(f^{1}_{\alpha^{\prime\prime}})s(e^{1}_{\nu})}.
    \end{equation*}
    Let 
    $(f^{1}_{\gamma^{\prime\prime}}, f^{2}_{\delta^{\prime\prime}})\in E_{1}\star E_{2}$ 
    be such that 
    $\theta(f^{1}_{\gamma^{\prime\prime}}, f^{2}_{\delta^{\prime\prime}})
    =(f^{2}_{\mu^{\prime\prime}},f^{1}_{\alpha^{\prime\prime}})$. 
    Then $\theta$ being a bijection, 
    $(f^{1}_{\gamma^{\prime\prime}}, f^{2}_{\delta^{\prime\prime}})$ 
    vary over all of $E_{1}\star E_{2}$ as 
    ($f^{2}_{\mu^{\prime\prime}},f^{1}_{\alpha^{\prime\prime}}$) 
    vary over all of $E_{2}\star E_{1}$. 
    Therefore using the relations \eqref{newrel}, 
    the last summation becomes
    \begin{equation*}
	\quad\qquad
	\sum_{\mathclap{(f^{1}_{\gamma^{\prime\prime}},f^{2}_{\delta^{\prime\prime}})\in E_{1}\star E_{2}}} 
	q_{t(e^{1}_{\alpha^{\prime}})t(f^{1}_{\gamma^{\prime\prime}})}
	q_{s(e^{1}_{\alpha^{\prime}})s(f^{1}_{\gamma^{\prime\prime}})}
	q_{s(e^{2}_{\beta^{\prime}})s(f^{2}_{\delta^{\prime\prime}})}
	\ot 
	q_{t(f^{1}_{\gamma^{\prime\prime}})t(e^{1}_{\alpha})}
	q_{s(f^{1}_{\gamma^{\prime\prime}})s(e^{1}_{\alpha})}
	q_{s(f^{2}_{\delta^{\prime\prime}})s(e^{2}_{\beta^{\prime}})},
    \end{equation*} 
    which can be shown to be equal to 
    $\Delta
    \left(
    q_{t(e^{1}_{\alpha^{\prime}})t(e^{1}_{\alpha})}
    q_{s(e^{1}_{\alpha^{\prime}})s(e^{1}_{\alpha})}
    q_{s(e^{2}_{\beta^{\prime}})s(e^{2}_{\beta})}
    \right)$.
\end{proof}

The next theorem will allow us to define the notion of the 
quantum automorphism group of a $2$-graph. 
The proof is more or less standard and hence omitted. 

\begin{theorem}
  \label{equivalence}
    Let $(\mathcal{G}_{1},\mathcal{G}_{2},\theta)$ 
    and 
    $(\mathcal{G}_{1}^{\prime},\mathcal{G}_{2}^{\prime},\theta^{\prime})$ 
    be two equivalent triples in the sense of Definition \ref{equivalent_triple}. Then 
    \begin{displaymath}
        {\rm Aut}^{+}(\mathcal{G}_{1},\mathcal{G}_{2},\theta)
        \cong 
        {\rm Aut}^{+}(\mathcal{G}_{1}^{\prime},\mathcal{G}_{2}^{\prime},\theta^{\prime}).
    \end{displaymath}
    In fact, ${\rm Aut}^{+}(\mathcal{G}_{1},\mathcal{G}_{2},\theta)$ 
    can be obtained from 
    ${\rm Aut}^{+}(\mathcal{G}_{1}^{\prime},\mathcal{G}_{2}^{\prime},\theta^{\prime})$ 
    by a similarity transformation.
\end{theorem}

Now as the isomorphism class of a $2$-graph depends on 
the equivalence class of its defining triple (see the Appendix), 
the above theorem allows us to define the following. 

\begin{definition}
     Let $\Lambda$ be a $2$-graph with the defining triple 
     $(\mathcal{G}_{1},\mathcal{G}_{2},\theta)$. 
     Then the quantum automorphism group of $\Lambda$ 
     is defined to be the quantum automorphism group of 
     $(\mathcal{G}_{1},\mathcal{G}_{2},\theta)$ 
     as defined in Definition~\ref{triple_def}. 
     The quantum automorphism group will be denoted by either of the 
     notations ${\rm Aut}^{+}(\Lambda)$ or 
     ${\rm Aut}^{+}(\mathcal{G}_{1},\mathcal{G}_{2},\theta)$.  
\end{definition}

\section{Examples and Computations I}
  \label{sec:computation}

\subsection*{(a) Two copies of the same graph with trivial bijection} 
Recall the pullback of a $1$-graph $\mathcal{G}$ via the monoid morphism 
$f:\mathbb{N}^{2}\raro\mathbb{N}$ given by $f(m,n)=m+n$. 
The pullback $f^{\ast}(\mathcal{G})$ is the $2$-graph with the defining triple 
$(\mathcal{G}_{1},\mathcal{G}_{2}, \theta)$, where $\mathcal{G}_{l}$'s 
are two copies of the same graph $\mathcal{G}$ and the bijection 
$\theta$ between the composable pairs is the trivial bijection. 
Since $\theta$ is the trivial bijection, replacing 
$(e^{2}_{\mu}, e^{1}_{\nu})$ by 
$(e^{1}_{\alpha},e^{2}_{\beta})$ and 
$(e^{2}_{\mu^{\prime}}, e^{1}_{\nu^{\prime}})$ by 
$(e^{1}_{\alpha^{\prime}}, e^{2}_{\beta^{\prime}})$ in~\eqref{newrel}, 
we see that \eqref{newrel} are redundant. 
Also since the individual $1$-graphs are two copies of the same graph, 
the following is immediate:
\begin{displaymath}
    {\rm Aut}^{+}(f^{\ast}(\mathcal{G}))\cong {\rm Aut}^{+}(\mathcal{G}).
\end{displaymath}
Thus we see that the quantum automorphism group of any 
$1$-graph can be realized as the quantum automorphism group of a 
$2$-graph $\Lambda$. Consequently we always have examples of 
$2$-graphs admitting genuine quantum automorphism groups.

\subsection*{(b) Two different graphs with a nontrivial bijection}
We take the following \(1\)-graphs \(\mathcal{G}_1\) and 
\(\mathcal{G}_2\) as the constituent graphs to construct the 
\(2\)-graph whose \(1\)-skeleton is given as follows:
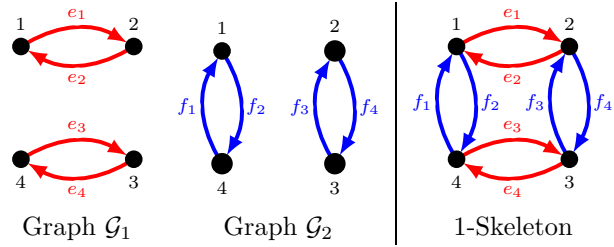
\begin{figure}[h]
  \begin{minipage}[h]{0.7\textwidth}
    \begin{tabular}{@{}cc|c@{}}
      \begin{tikzpicture}
        \Vertex[label=$1$,position=above,shape=circle,size=0.05,color=black]{1}
  	\Vertex[x=1.5,label=$2$,position=above,shape=circle, size=0.05,color=black]{2}
  	\Vertex[,y=-1.5,label=$4$,position=below,shape=circle, size=0.05,color=black]{4}
  	\Vertex[x=1.5,y=-1.5,label=$3$,position=below,shape=circle, size=0.05,color=black]{3}

  	\Edge[color=red,bend=30,label=$e_1$,position=above,Direct](1)(2);
  	\Edge[color=red,bend=30,label=$e_2$,position=below,Direct](2)(1);
  	\Edge[color=red,bend=30,label=$e_4$,position=below,Direct](3)(4);
  	\Edge[color=red,bend=30,label=$e_3$,position=above,Direct](4)(3);
      \end{tikzpicture}
&
      \begin{tikzpicture}
  	\Vertex[x=2.5,label=$1$,position=above,shape=circle, size=0.05,color=black]{5}
   	\Vertex[x=4,label=$2$,position=above,shape=circle, size=0.25,color=black]{6}
   	\Vertex[x=2.5, y=-1.5,label=$4$,position=below,shape=circle, size=0.25,color=black]{7}
   	\Vertex[x=4,y=-1.5,label=$3$,position=below,shape=circle, size=0.25,color=black]{8}

   	\Edge[color=blue,bend=30,label=$f_2$,position=right,Direct](5)(7);
   	\Edge[color=blue,bend=30,label=$f_1$,position=left,Direct](7)(5);
   	\Edge[color=blue,bend=30,label=$f_4$,position=right,Direct](6)(8);
   	\Edge[color=blue,bend=30,label=$f_3$,position=left,Direct](8)(6);
      \end{tikzpicture}
&
      \begin{tikzpicture}
	\Vertex[label=$1$,position=above,shape=circle,size=0.05,color=black]{1}
  	\Vertex[x=1.5,label=$2$,position=above,shape=circle, size=0.05,color=black]{2}
  	\Vertex[,y=-1.5,label=$4$,position=below,shape=circle, size=0.05,color=black]{4}
  	\Vertex[x=1.5,y=-1.5,label=$3$,position=below,shape=circle, size=0.05,color=black]{3}

  	\Edge[color=red,bend=30,label=$e_1$,position=above,Direct](1)(2);
  	\Edge[color=red,bend=30,label=$e_2$,position=below,Direct](2)(1);
  	\Edge[color=red,bend=30,label=$e_4$,position=below,Direct](3)(4);
  	\Edge[color=red,bend=30,label=$e_3$,position=above,Direct](4)(3);
  	\Edge[color=blue,bend=30,label=$f_2$,position=right,Direct](1)(4);
   	\Edge[color=blue,bend=30,label=$f_1$,position=left,Direct](4)(1);
   	\Edge[color=blue,bend=30,label=$f_4$,position=right,Direct](2)(3);
   	\Edge[color=blue,bend=30,label=$f_3$,position=left,Direct](3)(2);
      \end{tikzpicture}
\\
	Graph \(\mathcal{G}_1\) & Graph \(\mathcal{G}_2\) & 1-Skeleton
    \end{tabular}
  \end{minipage}
\end{figure}

Denote the edge sets of $\mathcal{G}_{1}, \mathcal{G}_{2}$ by 
$E_{1}, E_{2}$ respectively. Then the composable pairs are given by 
\begin{align*}
    E_1\star E_2 &=\{(e_1,f_1),(e_2,f_3),(e_3,f_2),(e_4,f_4)\},\\
    E_2\star E_1 &=\{(f_1,e_4),(f_2,e_2),(f_3,e_3),(f_4,e_1)\}.
\end{align*}
We take the following choice of $\theta$: 
\begin{align*}
	&(e_1,f_1)\mapsto(f_3,e_3) \\
	&(e_2,f_3)\mapsto(f_1,e_4) \\
	&(e_3,f_2)\mapsto(f_4,e_1) \\
	&(e_4,f_4)\mapsto(f_2,e_2) 
\end{align*}
In fact, the reader can verify that in this example there is a 
unique choice of $\theta$ given as above. 
With the usual notations, the action $\alpha$ of 
${\rm Aut}^{+}(\mathcal{G}_{1},\mathcal{G}_{2},\theta)$ on $C(V)$ 
is given by $\alpha(\delta_{i})=\sum_{j=1}^{4}\delta_{j}\ot q_{ji}$.  
A simple calculation reveals that the commutation of the unitary 
$((q_{ij}))$ with the vertex matrices impose the following 
relations among the generators $q_{ij}$'s: 
\begin{eqnarray}
	\label{exrel1}
	&& q_{11}=q_{22}=q_{33}=q_{44}; \ q_{12}=q_{21}=q_{34}=q_{43};\\
	\label{exrel2}
	&& q_{14}=q_{23}=q_{32}=q_{41}; \ q_{13}=q_{31}=q_{24}=q_{42}.
\end{eqnarray}
We claim that the relations~\eqref{newrel} are redundant for the graph 
$(\mathcal{G}_{1},\mathcal{G}_{2},\theta)$. 
We shall prove a somewhat stronger result: 
For any choice of edges from $\mathcal{G}_{2}$, 
the first two terms of the LHS of~\eqref{newrel} are equal i.e. 
$q_{t(f_{i})t(f_{j})}q_{s(f_{i})s(f_{j})}= q_{t(f_{i})t(f_{j})}$ for all $i,j=1,2,3,4$. 
This can be checked by direct calculation using the relations~\eqref{exrel1} and~\eqref{exrel2}. 
For example: 
\begin{eqnarray*}
	&& q_{t(f_{1})t(f_{2})}q_{s(f_{1})s(f_{2})}=q_{14}q_{41}=q_{14};\\
	&& q_{t(f_{1})t(f_{1})}q_{s(f_{1})s(f_{1})}=q_{11}q_{44}=q_{44}.
\end{eqnarray*} 
Similarly, one can prove that for $i,j=1,2,3,4$, 
$q_{t(e_{i})t(e_{j})}q_{s(e_{i})s(e_{j})}=q_{t(e_{i})t(e_{j})}$ 
so that, in particular, for all 
$(e^{1}_{\alpha},e^{2}_{\beta})\in E_{1}\star E_{2}$ 
and all 
$(e^{2}_{\mu^{\prime}},e^{1}_{\nu^{\prime}})\in E_{2}\star E_{1}$ 
with 
$\theta(e^{1}_{\alpha},e^{2}_{\beta})=(e^{2}_{\mu},e^{1}_{\nu})$ 
and 
$\theta(e^{1}_{\alpha^{\prime}},e^{2}_{\beta^{\prime}})
=(e^{2}_{\mu^{\prime}},e^{1}_{\nu^{\prime}})$,
\begin{align*}
	&q_{t(e^{2}_{\mu^{\prime}})t(e^{2}_{\mu})}
	  q_{s(e^{2}_{\mu^{\prime}})s(e^{2}_{\mu})}
	  q_{s(e^{1}_{\nu^{\prime}})s(e^{1}_{\nu})}\\
	&\quad= 
	  q_{t(e^{2}_{\mu^{\prime}})t(e^{2}_{\mu})}
	  q_{s(e^{1}_{\nu^{\prime}})s(e^{1}_{\nu})}\\
	&\quad= 
	  q_{t(e^{1}_{\alpha^{\prime}})t(e^{1}_{\alpha})}
	  q_{s(e^{2}_{\beta^{\prime}})s(e^{2}_{\beta})}\\
	&\quad= 
	  q_{t(e^{1}_{\alpha^{\prime}})t(e^{1}_{\alpha})}
	  q_{s(e^{1}_{\alpha^{\prime}})s(e^{1}_{\alpha})}
	  q_{s(e^{2}_{\beta^{\prime}})s(e^{2}_{\beta})}.
\end{align*}
Hence the fundamental unitary matrix of 
${\rm Aut}^{+}(\mathcal{G}_{1},\mathcal{G}_{2},\theta)$ reduces to
\begin{center} 
$\begin{pmatrix}
	q_{1} \ q_{2} \ q_{3} \ q_{4}\\
	q_{2} \ q_{1} \ q_{4} \ q_{3}\\
	q_{3} \ q_{4} \ q_{1} \ q_{2}\\
	q_{4} \ q_{3} \ q_{2} \ q_{1}\\
	\end{pmatrix}$,
\end{center}
where $q_{i}$'s are mutually orthogonal projections such that $\sum_{i=1}^{4}q_{i}=1$. 
Now it is easy to check that the above fundamental unitary 
corresponds to the CQG $C^{\ast}(\mathbb{Z}_{2}\times\mathbb{Z}_{2})$. 
Therefore, 
${\rm Aut}^{+}(\mathcal{G}_{1},\mathcal{G}_{2},\theta)$ 
is isomorphic to the classical group 
$\mathbb{Z}_{2}\times\mathbb{Z}_{2}$.

\subsection*{(c) Two copies of the graph $\mathcal{G}$ where 
			$\mathcal{G}$ is the complete graph on four vertices}
    \begin{center}    
     \begin{tikzpicture}[scale=.5]
      	\Vertex[label=$1$,position=above,shape=circle, size=0.25,color=black]{1}
	\Vertex[x=6,label=$2$,position=above,shape=circle, size=0.25,color=black]{2}
	\Vertex[y=-6,label=$3$,position=below,shape=circle, size=0.25,color=black]{3}
	\Vertex[x=6,y=-6,label=$4$,position=below,shape=circle, size=0.25,color=black]{4}

	\Edge[color=red,bend=20,label=$e_3$,Direct](1)(2);
	\Edge[color=red,label=$e_4$,Direct](2)(1);
	\Edge[color=red,label=$e_2$,Direct](1)(3);
	\Edge[color=red,bend=20,label=$e_1$,Direct](3)(1);
	\Edge[color=red,label=$e_8$,Direct](3)(4);
	\Edge[color=red,bend=20,label=$e_7$,Direct](4)(3);
	\Edge[color=red,bend=20,label=$e_6$,Direct](2)(4);
	\Edge[color=red,label=$e_5$,Direct](4)(2);
	\Edge[color=red,bend=20,label=$e_{12}$,Direct](4)(1);
	\Edge[color=red,bend=20,label=$e_{11}$,Direct](1)(4);
	\Edge[color=red,bend=20,label=$e_{10}$,Direct](2)(3);
	\Edge[color=red,bend=20,label=$e_9$,Direct](3)(2);
     \end{tikzpicture}
\end{center}
 The composable pairs are given by\\
\[
E\ \star\ E=\left\{
\begin{matrix}
&(e_1,e_2),&(e_1,e_{10}),&(e_1,e_7),&(e_2,e_4),&(e_2,e_{12}), &(e_2,e_1),\\
&(e_3,e_4),&(e_3,e_{12}),&(e_3,e_1),&(e_4,e_3),&(e_4,e_9),&(e_4,e_5),\\
&(e_5,e_6),&(e_5,e_{11}),&(e_5,e_8),&(e_6,e_5),&(e_6,e_3),&(e_6,e_9),\\
&(e_7,e_6),&(e_7,e_{11}),&(e_7,e_8),&(e_8,e_2),&(e_8,e_10),&(e_8,e_7),\\
&(e_9,e_2),&(e_9,e_{10}),&(e_9,e_7),&(e_{10},e_3),&(e_{10},e_9),&(e_{10},f5),\\
&(e_{11},e_1),&(e_{11},e_4),&(e_{11},e_{12}),&(e_{12},e_6),&(e_{12},e_{11}),&(e_{12},e_8)
\end{matrix}
\right\}.
\]
We choose the map $\theta$ to be the following: 
\begin{align*}
  &(e_1,e_2)\mapsto(e_4,e_3), \\
  &(e_4,e_3)\mapsto(e_1,e_2), \\
  &(e_8,e_7)\mapsto(e_6,e_5), \\
  &(e_6,e_5)\mapsto(e_8,e_7), 
\end{align*}
keeping rest of the elements of \(E\star E\) fixed. 
As usual, we write the action of 
${\rm Aut}^{+}(\mathcal{G},\mathcal{G},\theta)$ 
on the set of vertices by
\begin{displaymath}
    \alpha(\delta_{i})=\sum_{j=1}^{4}\delta_{j}\ot q_{ji},\text{ for } i=1,\ldots,4.
\end{displaymath}
We assert that
\begin{enumerate}
    \item \(q_{13}=q_{12}=0\) (implies \(q_{31}=q_{21}=0\)),
    \item \(q_{43}=q_{42}=0\) (implies \(q_{34}=q_{24}=0\)).
\end{enumerate}
To prove our assertions, we shall first establish the equality \(q_{13}=q_{12}\). 
By subsequently multiplying both sides of this equation by \(q_{13}\), 
we deduce that \(q_{13}\) equals zero.
To accomplish this, 
we show that \(q_{i1}q_{13}q_{j1}=q_{i1}q_{12}q_{j1}\) for all \(i,j=1,2,\ldots,4\). 
Recalling from equation \eqref{newrel} for \(\theta(e,f)=(f',e')\) 
and \(\theta(\mu,\nu)=(\nu',\mu')\) we have
\begin{equation*}
    q_{t(\mu)t(e)}q_{s(\mu)s(e)}q_{s(\nu)s(f)}
    =
    q_{t(\nu')t(f')}q_{s(\nu')s(f')}q_{s(\mu')s(e')}.
\end{equation*}
For \((e,f)=(e_1,e_2)\) and \((f',e')=(e_4,e_3)\) the above equation reduces to 
\[
  q_{t(\mu)1}q_{s(\mu)3}q_{s(\nu)1}=q_{t(\nu')1}q_{s(\nu')2}q_{t(\mu')1}.
\]
We have the following equations for different choices of 
\((\mu,\nu)\in\{(e_3,e_4),(e_3,e_1),\\
(e_3,e_{12}),(e_2,e_4),(e_2,e_1),(e_2,e_{12}),(e_{11},e_{4}),(e_{11},e_1),(e_{11},e_{12})\}\): 
\begin{equation*}
  \begin{matrix}
  q_{21}q_{13}q_{21}=q_{21}q_{12}q_{21}, 
  & 
  q_{21}q_{13}q_{31}=q_{21}q_{12}q_{31}, 
  & 
  q_{21}q_{13}q_{41}=q_{21}q_{12}q_{41},\\
  q_{31}q_{13}q_{21}=q_{31}q_{12}q_{21}, 
  & 
  q_{31}q_{13}q_{31}=q_{31}q_{12}q_{31}, 
  & 
  q_{31}q_{13}q_{41}=q_{31}q_{12}q_{41}, \\
  q_{41}q_{13}q_{21}=q_{41}q_{12}q_{21}, 
  & 
  q_{41}q_{13}q_{31}=q_{41}q_{12}q_{31}, 
  & 
  q_{41}q_{13}q_{41}=q_{41}q_{12}q_{41},
  \end{matrix}
\end{equation*}
whereas the other equations follow automatically as both sides of 
the equations are $0$ using the orthogonality properties of $q_{ij}$'s.
Adding up all the equations above we get 
\[
\sum_{i,j=1}^{4}q_{i1}q_{13}q_{j1}=\sum_{i,j=1}^{4}q_{i1}q_{12}q_{j1} 
\]
which clearly shows that \(q_{13}=q_{12}\).
A similar calculation shows that \(q_{43}=q_{42}\).
Therefore, the fundamental matrix takes the following form:
\begin{equation*}
    \left(
      \begin{matrix}
          q_{11} & 0 & 0 & q_{14} \\
          0     & q_{22} & q_{23} & 0 \\
          0     & q_{32} & q_{33} & 0 \\
          q_{41} & 0 & 0 & q_{44}
      \end{matrix}
    \right).
\end{equation*}
Assume that \(q_{11}=p\) and \(q_{22}=q\), where \(p\) and \(q\) are projections.  
Then the magic condition forces 
\(q_{41}=1-p=q_{14}, q_{44}=p, q_{32}=1-q=q_{23}, q_{33}=q\).
In the above calculation putting \((\mu,\nu)=(e_8,f_2)=(\nu',\mu')\) 
we see that
\[
   q_{41}q_{33}q_{11}=q_{41}q_{32}q_{11}.
\] 
Hence we have \((1-p)qp = (1-p)(1-q)p\) which implies \(qp=pqp\).
Consequently, \(qp=pqp=(qp)^*=pq\) shows the \(C^*\)-algebra generated 
by \(q_{ij}\)'s is commutative. Therefore, the quantum automorphism group 
is actually classical and in fact a subgroup of $\mathbb{Z}_{2}\times\mathbb{Z}_{2}$. 
Now it can be checked easily that none of the order 
$2$ elements in $S_{4}$ generating $\mathbb{Z}_{2}\times\mathbb{Z}_{2}$ 
belongs to the symmetry group. 
Hence the quantum automorphism group is not only classical, but also trivial. 

\section{Examples and computations II: Free product and free wreath product} \label{wreath product}

We begin this section by introducing a notion of quantum equivalence 
between a pair of triples \((\mathcal{G}_{1},\mathcal{G}_{2},\theta)\) 
and \((\mathcal{G}^{\prime}_{1},\mathcal{G}^{\prime}_{2},\theta^{\prime})\). 
We shall denote the sets of vertices of 
\((\mathcal{G}_{1},\mathcal{G}_{2},\theta)\) and 
\((\mathcal{G}^{\prime}_{1},\mathcal{G}^{\prime}_{2},\theta^{\prime})\)  
by $V,V^{\prime}$ respectively; the sets of edges of 
$\mathcal{G}_{1},\mathcal{G}_{2},\mathcal{G}^{\prime}_{1},\mathcal{G}^{\prime}_{2}$ 
by $E_{1},E_{2},E^{\prime}_{1},E^{\prime}_{2}$ respectively; 
the vertex matrices of 
$\mathcal{G}_{1},\mathcal{G}_{2},\mathcal{G}^{\prime}_{1},\mathcal{G}^{\prime}_{2}$ 
by $\varepsilon_{1},\varepsilon_{2},\varepsilon^{\prime}_{1}, \varepsilon^{\prime}_{2}$ 
respectively.

\begin{definition}
    \label{qiso}
    Two triples $(\mathcal{G}_{1},\mathcal{G}_{2},\theta)$ and 
    $(\mathcal{G}^{\prime}_{1},\mathcal{G}^{\prime}_{2},\theta^{\prime})$ 
    are said to be quantum equivalent 
    (to be denoted by $(\mathcal{G}_{1},\mathcal{G}_{2},\theta)
    \sim^{q}
    (\mathcal{G}^{\prime}_{1},\mathcal{G}^{\prime}_{2},\theta^{\prime})$) 
    if there is a $C^{\ast}$-algebra $\cla$ containing elements 
    $\{u_{ij}\}_{i\in V,j\in V^{\prime}}$ 
    such that
    \begin{enumerate}[(1)]
        \item the matrix $U=((u_{ij}))_{i\in V, j\in V^{\prime}}$ is a magic unitary;
        \item $U\varepsilon_{1}^{\prime}=\varepsilon_{1}U$, 
        		$U\varepsilon_{2}^{\prime}=\varepsilon_{2}U$;
        \item for all 
        		$(e^{1}_{\alpha^{\prime}},e^{2}_{\beta^{\prime}})\in E_{1}^{\prime}\star E_{2}^{\prime}$ 
            	and $(e^{2}_{\mu},e^{1}_{\nu})\in E_{2}\star E_{1}$ with 
		$\theta^{\prime}(e^{1}_{\alpha^{\prime}},e^{2}_{\beta^{\prime}})
		=(e^{2}_{\mu^{\prime}},e^{1}_{\nu^{\prime}})$ and 
		$\theta(e^{1}_{\alpha},e^{2}_{\beta})=(e^{2}_{\mu},e^{1}_{\nu})$
    		\begin{equation}
    		  \label{genrelqequivalent}
			u_{t(e^{2}_{\mu})t(e^{2}_{\mu^{\prime}})}
			u_{s(e^{2}_{\mu})s(e^{2}_{\mu^{\prime}})}
			u_{s(e^{1}_{\nu})s(e^{1}_{\nu^{\prime}})}
			=
			u_{t(e^{1}_{\alpha})t(e^{1}_{\alpha^{\prime}})}
			u_{s(e^{1}_{\alpha})s(e^{1}_{\alpha^{\prime}})}
			u_{s(e^{2}_{\beta})s(e^{2}_{\beta^{\prime}})}.
    		\end{equation}
    \end{enumerate}
\end{definition}

\begin{remark}
    Note that if the elements $\{u_{ij}\}_{i\in V,j\in V^{\prime}}$ in $\cla$ 
    satisfy conditions (1) and (2) of Definition \ref{qiso}, 
    then there is a well defined $C^{\ast}$-homomorphism 
    $\alpha:C(V^{\prime})\rightarrow C(V)\ot\cla$ given by 
    $\alpha(\delta_{j})=\sum_{i\in V}\delta_{i}\ot u_{ij}$ for 
    $j\in V^{\prime}$ where $\{\delta_{i}\}_{i\in V},\{\delta_{j}\}_{j\in V^{\prime}}$ 
    are canonical basis for $C(V)$ and $C(V^{\prime})$ respectively. 
    Consequently there is a well defined $\mathbb{C}$-linear map 
    $\alpha^{\ot 4}:C(V^{\prime}\times V^{\prime}\times V^{\prime} \times V^{\prime})
    \rightarrow 
    C(V\times V\times V\times V)\ot\cla$. 
    Also the map $\alpha^{\ot 4}$ maps $C(E_{1}^{\prime}\star E_{2}^{\prime})$ into 
    $C(E_{1}\star E_{2})\otimes\cla$ and $C(E_{2}^{\prime}\star E_{1}^{\prime})$ into 
    $C(E_{2}\star E_{1})\otimes\cla$. 
    Then condition (3) is equivalent to the following commutative diagram:
    \begin{center}
	\begin{tikzcd}
	C(E_{1}^{\prime}\star E_{2}^{\prime}) \arrow{r}{\alpha^{\ot 4}} \arrow[swap]{d}{\theta_{\ast}^{\prime}} 
	& 
	C(E_{1}\star E_{2})\ot \cla \arrow{d}{\theta_{\ast}\ot {\rm id}} \\
	C(E_{2}^{\prime}\star E_{1}^{\prime}) \arrow{r}{\alpha^{\ot 4}} 
	& 
	C(E_{2}\star E_{1})\ot \cla
	\end{tikzcd}
    \end{center}
    With appropriate choice of a basis, the above commutative diagram is equivalent to
    \begin{equation}
    \label{coordinate free}
        (U\ot U\ot U\ot U)\circ\theta^{\prime}=\theta\circ(U\ot U\ot U\ot U),
    \end{equation}
    where with abuse of notation, we have denoted the matrix of the linear maps 
    $\theta_{\ast}$ and $\theta^{\prime}_{\ast}$ by $\theta$ and $\theta^{\prime}$ respectively. 
\end{remark}

\begin{lemma}
    \label{well defined}
    $\sim^{q}$ is an equivalence relation.
\end{lemma}

\begin{proof}
    (1) Reflexivity: For a triple \((\mathcal{G}_{1},\mathcal{G}_{2},\theta)\) 
    with a vertex set $V$, take the $C^{\ast}$-algebra to be $\mathbb{C}$, 
    and the elements $\{u_{ij}\}_{i,j\in V}$ to be $u_{ii}=1$ and $u_{ij}=0$ for $i\neq j$. 
    
    \medskip
    
    \noindent (2) Symmetry: Suppose 
    $(\mathcal{G}_{1},\mathcal{G}_{2},\theta)
    \sim^{q}
    (\mathcal{G}^{\prime}_{1},\mathcal{G}_{2}^{\prime},\theta^{\prime})$. 
    Let $\cla$ be the $C^{\ast}$-algebra with elements $\{u_{ij}\}_{i\in V, j\in V^{\prime}}$ 
    establishing the equivalence. To show that 
    $(\mathcal{G}^{\prime}_{1},\mathcal{G}^{\prime}_{2},\theta^{\prime})
    \sim^{q}
    (\mathcal{G}_{1},\mathcal{G}_{2},\theta)$ 
    we take the $C^{\ast}$-algebra to be $\cla$ 
    and consider the set of elements 
    $\{v_{ji}\}_{j\in V^{\prime},i\in V}$ where \(v_{ji}=u_{ij}\).

    \begin{enumerate}[(a)]
    	\item $\sum_{i\in V}v_{ji}
		=\sum_{i\in V}u_{ij}
		=1$ 
		and 
		$\sum_{j\in V^{\prime}}v_{ji}
		=\sum_{j\in V^{\prime}}u_{ij}
		=1$ 
		for every $j\in V^{\prime}, i\in V$.
   	\item  For $l=1,2$ using symmetry of the vertex matrices 
		of $\clg_{l},\clg_{l}^{\prime}$,
     		\begin{align*}
        		    (\varepsilon_{l}V)_{ji} 
		    &= \sum\limits_{k\in V^{\prime}}(\varepsilon_{l})_{jk}v_{ki} \\
        		    &=  \sum\limits_{k\in V^{\prime}}(\varepsilon_{l})_{jk}u_{ik} \\
        		    &= \sum\limits_{k\in V^{\prime}}u_{ik}(\varepsilon_{l})_{kj} \\
        		    &= (U\varepsilon_{l})_{ij} \\
        		    &= (\varepsilon_{l}^{\prime}U)_{ij} \\
        		    &= \sum\limits_{k\in V}(\varepsilon_{l}^{\prime})_{ik}u_{kj} \\
        		    &= \sum\limits_{k\in V}v_{jk}(\varepsilon_{l}^{\prime})_{ki} \\
        		    &= (V\varepsilon_{l}^{\prime})_{ji}. 
    		\end{align*}
   	\item  As for the third condition of Definition \ref{qiso}, 
		let $\theta(e^{1}_{\alpha},e^{2}_{\beta})=(e^{2}_{\mu},e^{1}_{\nu})$ 
		and $\theta^{\prime}(e^{1}_{\alpha^{\prime}},e^{2}_{\beta^{\prime}})
		=(e^{2}_{\mu^{\prime}},e^{1}_{\nu^{\prime}})$. Then 
    		\begin{eqnarray*}
		    v_{t(e^{2}_{\mu^{\prime}})t(e^{2}_{\mu})}
		    v_{s(e^{2}_{\mu^{\prime}})s(e^{2}_{\mu})}
		    v_{s(e^{1}_{\nu^{\prime}})s(e^{1}_{\nu})}
		    &=&
		    u_{t(e^{2}_{\mu})t(e^{2}_{\mu^{\prime}})}
		    u_{s(e^{2}_{\mu})s(e^{2}_{\mu^{\prime}})}
		    u_{s(e^{1}_{\nu})s(e^{1}_{\nu^{\prime}})}; \\ 
		    v_{t(e^{1}_{\alpha^{\prime}})t(e^{1}_{\alpha})}
		    v_{s(e^{1}_{\alpha^{\prime}})s(e^{1}_{\alpha})}
		    v_{s(e^{2}_{\beta^{\prime}})s(e^{2}_{\beta})}
		    &=& 
		    u_{t(e^{1}_{\alpha})t(e^{1}_{\alpha^{\prime}})}
		    u_{s(e^{1}_{\alpha})s(e^{1}_{\alpha^{\prime}})}
		    u_{s(e^{2}_{\beta})s(e^{2}_{\beta^{\prime}})}.
    		\end{eqnarray*}    
		Now the right hand sides of the above two equations are equal as 
		$(\mathcal{G}_{1},\mathcal{G}_{2},\theta)
		\sim^{q}
		(\mathcal{G}^{\prime}_{1},\mathcal{G}_{2}^{\prime},\theta^{\prime})$. 
		Therefore the left hand sides are also equal proving that 
		$(\mathcal{G}^{\prime}_{1},\mathcal{G}^{\prime}_{2},\theta^{\prime})
		\sim^{q}
		(\mathcal{G}_{1},\mathcal{G}_{2},\theta)$.
	\end{enumerate}
	
    \medskip
    
    \noindent (3) Transitivity: Suppose 
    $(\mathcal{G}_{1},\mathcal{G}_{2},\theta)
    \sim^{q}
    (\mathcal{G}^{\prime}_{1},\mathcal{G}_{2}^{\prime},\theta^{\prime})$
     and 
     $(\mathcal{G}^{\prime}_{1},\mathcal{G}^{\prime}_{2},\theta^{\prime})
     \sim^{q}
     (\mathcal{G}^{\prime\prime}_{1},\mathcal{G}_{2}^{\prime\prime},\theta^{\prime\prime})$; 
     let $\cla$ be the $C^{\ast}$-algebra with elements $\{u_{ij}\}_{i\in V,j\in V^{\prime}}$ 
     establishing the quantum equivalence 
     $(\mathcal{G}_{1},\mathcal{G}_{2},\theta)
     \sim^{q}
     (\mathcal{G}^{\prime}_{1},\mathcal{G}_{2}^{\prime},\theta^{\prime})$ 
     and 
     $\clb$ be the $C^{\ast}$-algebra with elements 
     $\{v_{ij}\}_{i\in V^{\prime},j\in V^{\prime\prime}}$ 
     establishing the quantum equivalence 
     $(\mathcal{G}^{\prime}_{1},\mathcal{G}^{\prime}_{2},\theta^{\prime})
     \sim^{q}
     (\mathcal{G}^{\prime\prime}_{1},\mathcal{G}_{2}^{\prime\prime},\theta^{\prime\prime})$. 
     We denote the vertex matrices of 
     $\mathcal{G}_{l},\mathcal{G}_{l}^{\prime},\clg_{l}^{\prime\prime}$
      by 
      $\varepsilon_{l},\varepsilon_{l}^{\prime},\varepsilon^{\prime\prime}_{l}$ 
      respectively for $l=1,2$. 
      Consider a $C^{\ast}$-algebra $\clc:=\cla\hat{\ot}\clb$ 
      (one can take any $C^{\ast}$-algebraic completion of 
      $\cla\ot_{\rm alg}\clb$) with elements 
      $\{w_{ij}\}_{i\in V,j\in V^{\prime\prime}}$ defined by
      \begin{displaymath}
    	w_{ij}:=\sum_{k\in V^{\prime}}u_{ik}\ot v_{kj}.
      \end{displaymath}
     If we denote the $\clc$-valued matrix $((w_{ij}))_{i\in V,j\in V^{\prime\prime}}$ 
     by $W$, then it is easy to check that $W$ is again a magic unitary. 
     As for the other conditions, for $l=1,2$,
    \begin{align*}
   	(\varepsilon_{l}W)_{ij}&=\sum_{m\in V^{\prime}}(\varepsilon_{l})_{im}w_{mj}\\
   	&= \sum_{m,k\in V^{\prime}}(\varepsilon_{l})_{im}u_{mk}\ot v_{kj}\\
   	&= \sum_{m,k\in V^{\prime}}u_{im}\ot(\varepsilon^{\prime}_{l})_{mk}v_{kj}\\
   	&= \sum_{m,k\in V^{\prime}}u_{im}\ot v_{mk}(\varepsilon^{\prime\prime}_{l})_{kj}\\
   	&=\sum_{k\in V^{\prime}}w_{ik}(\varepsilon^{\prime\prime})_{kj}\\
   	&= (W\varepsilon^{\prime\prime}_{l})_{ij}.
    \end{align*}
    Therefore, $\varepsilon_{l}W=W\varepsilon^{\prime\prime}_{l}$ for $l=1,2$. 
    The third condition of Definition \ref{qiso} can be verified 
    along the lines of the proof of the fact that the ideal generated by relations~\eqref{newrel} 
    is a Woronowicz ideal and we leave it to the reader. 
\end{proof}

\begin{lemma}
    \label{eqqeq}
    Two equivalent triples are quantum equivalent.
\end{lemma}

\begin{proof}
    For two equivalent triples with vertex sets $V, V^{\prime}$, 
    let $T:V\raro V^{\prime}$ be the bijection. 
    Then  one can take the $C^{\ast}$-algebra in the definition of 
    quantum equivalence to be $\mathbb{C}$ and the elements 
    $\{t_{ij}\}_{i\in V, j\in V^{\prime}}\in\mathbb{C}$ to be the entries 
    of the permutation matrix representing the bijection $T$. 
    Then by the definition of equivalence of triples, $\{t_{ij}\}$'s 
    satisfy conditions (1), (2) of Definition \ref{qiso} and the the 
    matrix $((t_{ij}))$ satisfies Equation \eqref{coordinate free}. 
\end{proof}

Thanks to Lemma \ref{well defined} and Lemma \ref{eqqeq} we can 
define the notion of quantum isomorphism between a pair of $2$-graphs.

\begin{definition}
    \label{qiso2graph}
    Let $(\Lambda,d)$ and $(\Lambda^{\prime},d^{\prime})$ be two 
    $2$-graphs with the defining triples 
    $(\mathcal{G}_{1},\mathcal{G}_{2},\theta)$ and 
    $(\mathcal{G}^{\prime}_{1},\mathcal{G}^{\prime}_{2},\theta^{\prime})$ respectively. 
    Then $(\Lambda,d)$ and $(\Lambda^{\prime},d^{\prime})$ 
    are defined to be quantum isomorphic if 
    $(\mathcal{G}_{1},\mathcal{G}_{2},\theta)
    \sim^{q} 
    (\mathcal{G}^{\prime}_{1},\mathcal{G}^{\prime}_{2},\theta^{\prime})$. 
\end{definition}

\begin{remark}
    \label{equiv_qequiv}
     The definition of quantum equivalence opens up the question 
     to find an example of a pair of triples which are not classically equivalent, 
     but quantum equivalent. We do not pursue this question in this paper.
\end{remark}

Now observe that by the definition of quantum equivalence, 
a pair of triples with vertex sets $V$ and $V^{\prime}$ such that 
$|V|\neq |V^{\prime}|$, can not be quantum equivalent. 
Therefore only pairs with same cardinality of vertex sets 
are of real interest in the context of quantum equivalence. 
From now on we shall have the following assumptions on 
the triples unless mentioned otherwise:\\
{\bf Standing assumption}: (i) Given a pair of triples 
$(\clg_{1},\clg_{2},\theta)$ and 
$(\clg_{1}^{\prime},\clg_{2}^{\prime},\theta^{\prime})$, 
the underlying vertex sets have same cardinality i.e. $|V|=|V^{\prime}|$.\\
(ii) All the graphs $\clg_{1},\clg_{2},\clg_{1}^{\prime},\clg_{2}^{\prime}$ 
are connected, bidirected, without multiple edges or loops.

 Given a pair of triples 
$(\mathcal{G}_{1},\mathcal{G}_{2},\theta)$ and 
$(\mathcal{G}^{\prime}_{1},\mathcal{G}^{\prime}_{2},\theta^{\prime})$, 
we consider the triple $(\Gamma_{1},\Gamma_{2},\Theta)$, 
where $\Gamma_{l}$ are given by the disjoint unions 
$\mathcal{G}_{l}\sqcup\mathcal{G}^{\prime}_{l}$ for $l=1,2$. 
Then the common vertex set of $\Gamma_{1},\Gamma_{2}$ 
is given by $V\sqcup V^{\prime}$ and the composable pairs are 
given by $(E_{1}\star E_{2})\sqcup(E_{1}^{\prime}\star E_{2}^{\prime})$. 
Moreover, $\Theta$ is given by $\theta$ on $E_{1}\star E_{2}$ and 
$\theta^{\prime}$ on $E^{\prime}_{1}\star E^{\prime}_{2}$. 
Now we shall consider the quantum automorphism group 
of the triple $(\Gamma_{1},\Gamma_{2},\Theta)$ as it will be 
shown to be intimately connected to the question of quantum 
equivalence of the pair $(\mathcal{G}_{1},\mathcal{G}_{2},\theta)$ 
and $(\mathcal{G}_{1}^{\prime},\mathcal{G}^{\prime}_{2},\theta^{\prime})$. 
We remark that the graphs $\Gamma_{1}, \Gamma_{2}$ are no longer connected. 
But they are bidirected without multiple edges, loops or isolated vertices. 
Note that the fundamental unitary matrix of 
${\rm Aut}^{+}(\Gamma_{1},\Gamma_{2},\Theta)$ 
can be written as the following block matrix:

\begin{equation}
  \label{Fun-unit}
     \left( 
       \begin{BMAT}(e)[2pt,2cm,2cm]{c.c}{c.c}
            ((q_{ij}))_{i,j\in V} & ((q_{il}))_{i\in V,l\in V^{\prime}} \\ 
            ((q_{kj}))_{k\in V^{\prime},j\in V} & ((q_{kl}))_{k,l\in V^{\prime}} 
        \end{BMAT}
     \right)
\end{equation}
Consider the upper off diagonal block. For each $i\in V$, we define a projection 
$p_{i}:=\sum_{l\in V^{\prime}}q_{il}$ and for each $l\in V^{\prime}$, 
define $p_{l}^{\prime}:=\sum_{i\in V}q_{il}$. Then using the standing 
assumptions on the pair 
$(\mathcal{G}_{1},\mathcal{G}_{2},\theta), 
(\mathcal{G}_{1}^{\prime},\mathcal{G}_{2}^{\prime},\theta^{\prime})$, 
one can prove the following  
\begin{lemma}
    \label{common projection}
    For any $i_{1},i_{2}\in V$ and $l_{1}, l_{2}\in V^{\prime}$, 
    $p_{i_{1}}=p_{i_{2}}=p^{\prime}_{l_{1}}=p^{\prime}_{l_{2}}$.
\end{lemma}

\begin{proof}
    See the proof of~\cite{Qinfo}*{Theorem 4.4}.
\end{proof}

We denote any of the above projections by $p$. 
Then we extract the relevant part of~\cite{Qinfo}*{Theorem 4.4} 
for our purpose and state it as the following proposition:

\begin{proposition}
    \label{equivalent condition}
    For a pair of triples $(\mathcal{G}_{1},\mathcal{G}_{2},\theta)$ and 
    $(\mathcal{G}^{\prime}_{1},\mathcal{G}^{\prime}_{2},\theta^{\prime})$ 
    with standing assumptions, 
    $(\mathcal{G}_{1},\mathcal{G}_{2},\theta)
    \sim^{q}
    (\mathcal{G}^{\prime}_{1},\mathcal{G}^{\prime}_{2},\theta^{\prime})$ 
    if and only if $p\neq 0$.
\end{proposition}

\begin{proof}
    We essentially adapt and extend the proof of~\cite{Qinfo}*{Theorem 4.4}. 
    Suppose 
    $(\mathcal{G}_{1},\mathcal{G}_{2},\theta)
    \sim^{q}
    (\mathcal{G}^{\prime}_{1},\mathcal{G}^{\prime}_{2},\theta^{\prime})$. 
    Then by definition, there is a $C^{\ast}$-algebra $\cla$ containing elements 
    $\{u_{il}\}_{i\in V,l\in V^{\prime}}$ such that
    \begin{enumerate}[(i)]
        \item the matrix $U=((u_{il}))_{i\in V,l\in V^{\prime}}$ is a magic unitary;
        \item $\varepsilon_{1}U=U\varepsilon_{1}^{\prime}$, 
        		$\varepsilon_{2}U=U\varepsilon_{2}^{\prime}$;
        \item For all $(e^{1}_{\alpha^{\prime}},e^{2}_{\beta^{\prime}})\in E_{1}^{\prime}\star E_{2}^{\prime}$ 
        		and $(e^{2}_{\mu},e^{1}_{\nu})\in E_{2}\star E_{1}$ 
		with 
		$\theta^{\prime}(e^{1}_{\alpha^{\prime}},e^{2}_{\beta^{\prime}})
		=
		(e^{2}_{\mu^{\prime}},e^{1}_{\nu^{\prime}})$ 
		and 
		$\theta(e^{1}_{\alpha},e^{2}_{\beta})
		=
		(e^{2}_{\mu},e^{1}_{\nu})$
    		\begin{equation}
    		  \label{genrelqequivalent2}
		    u_{t(e^{2}_{\mu})t(e^{2}_{\mu^{\prime}})}
		    u_{s(e^{2}_{\mu})s(e^{2}_{\mu^{\prime}})}
		    u_{s(e^{1}_{\nu})s(e^{1}_{\nu^{\prime}})} 
		    = 
		    u_{t(e^{1}_{\alpha})t(e^{1}_{\alpha^{\prime}})}
		    u_{s(e^{1}_{\alpha})s(e^{1}_{\alpha^{\prime}})}
		    u_{s(e^{2}_{\beta})s(e^{2}_{\beta^{\prime}})}.
    		\end{equation}
    \end{enumerate}
    We define a representation 
    $\pi$ of ${\rm Aut}^{+}(\Gamma_{1},\Gamma_{2},\Theta)$ in $\cla$ 
    by using the universal property of ${\rm Aut}^{+}(\Gamma_{1},\Gamma_{2},\Theta)$. 
    Suppose \(W=\begin{pmatrix}
      	    0 & U\\
            U^{\top} & 0
  	\end{pmatrix}\).
    It is clear that $W$ is a magic unitary. Using the facts that 
    $\varepsilon_{l}U=U\varepsilon_{l}^{\prime}$ for $l=1,2$ 
    and following the same calculations done in the proof of symmetry 
    of the relation $\sim^{q}$, it is easy to see that  
    \begin{equation*}
      W
  	\begin{pmatrix}
      	    \varepsilon_l & 0\\
            0 & \varepsilon^{\prime}_l
  	\end{pmatrix}
  	=
 	 \begin{pmatrix}
    	    \varepsilon_l & 0\\
     	    0 & \varepsilon^{\prime}_l
  	\end{pmatrix}
      W;
    \end{equation*}
    We shall show that matrix entries of $W$ satisfy Equation~\eqref{newrel}. 
    By the definition of $\Theta$, there are three cases:
    \begin{itemize}
        \item $(e^{1}_{\alpha},e^{2}_{\beta})\in E_{1}\star E_{2}$, 
        		$(e^{2}_{\gamma},e^{1}_{\delta})\in E_{2}\star E_{1}$ 
		such that 
		$\theta(e^{1}_{\alpha},e^{2}_{\beta})
		=
		(e^{2}_{\mu},e^{1}_{\nu})\in E_{2}\star E_{1}$ 
		and 
		$\theta(e^{1}_{\zeta},e^{2}_{\eta})
		=
		(e^{2}_{\gamma},e^{1}_{\delta})\in E_{2}\star E_{1}$: 
		In this case both sides of Equation~\eqref{newrel} are zero.
	\item Similarly if the composable pairs come from the second triple, 
		both sides of Equation~\eqref{newrel} are zero.
	\item For $(e^{1}_{\alpha^{\prime}},e^{2}_{\beta^{\prime}})
		\in 
		E_{1}^{\prime}\star E_{2}^{\prime}$ 
		and 
		$(e^{2}_{\mu},e^{1}_{\nu})\in E_{2}\star E_{1}$ 
		with 
		$\theta^{\prime}(e^{1}_{\alpha^{\prime}},e^{2}_{\beta^{\prime}})
		=
		(e^{2}_{\mu^{\prime}},e^{1}_{\nu^{\prime}})$ 
		and 
		$\theta(e^{1}_{\alpha},e^{2}_{\beta})
		=
		(e^{2}_{\mu},e^{1}_{\nu})$, 
		the Equation~\eqref{newrel} are satisfied by Equation~\eqref{genrelqequivalent}.\\
		\indent Therefore, by the universal property of 
		${\rm Aut}^{+}(\Gamma_{1},\Gamma_{2},\Theta)$, 
		there is a well defined $C^{\ast}$-homomorphism 
		$\pi:{\rm Aut}^{+}(\Gamma_{1},\Gamma_{2},\Theta)\rightarrow\cla$ 
		such that $\pi(q_{il})=u_{il}$ for $i\in V, l\in V^{\prime}$. 
		Hence, $\pi(p)=1\in\cla$ proving that 
		$p\neq 0$ in ${\rm Aut}^{+}(\Gamma_{1},\Gamma_{2},\Theta)$.
    \end{itemize}

    \medskip
    
    Conversely, let $p\neq 0$. Then consider the $C^{\ast}$-subalgebra 
    $\cla$ of ${\rm Aut}^{+}(\Gamma_{1},\Gamma_{2},\Theta)$ generated 
    by the entries of the upper off diagonal entries i.e. $\{q_{il}\}_{i\in V,l\in V^{\prime}}$. 
    Observe that in $\cla$, $p$ acts as the identity element making 
    the matrix $Q=((q_{il}))_{i\in V,l\in V^{\prime}}$ a magic unitary. 
    The commutation of $U$ with the matrices 
    $\begin{pmatrix}
      \varepsilon_l & 0\\
      0 & \varepsilon^{\prime}_l
    \end{pmatrix}$ 
    implies that $\varepsilon_{l}Q=Q\varepsilon_{l}^{\prime}$ for $l=1,2$. 
    As $U$ is the fundamental unitary of 
    ${\rm Aut}^{+}(\Gamma_{1},\Gamma_{2},\Theta)$, 
    considering composable pairs from $E_{1}\star E_{2}$ and 
    $E_{1}^{\prime}\star E_{2}^{\prime}$, it is easy to see that the elements 
    $\{q_{il}\}_{i\in V, l\in V^{\prime}}$ satisfy the relations~\eqref{genrelqequivalent}. 
    Therefore, by definition, 
    $(\clg_{1},\clg_{2},\theta)\sim^{q}(\clg_{1}^{\prime},\clg_{2}^{\prime},\theta^{\prime})$. 
  \end{proof}
  
Again let $(\clg_{1},\clg_{2},\theta)$ and 
$(\clg_{1}^{\prime},\clg_{2}^{\prime},\theta^{\prime})$ be two triples. 
Then recall the triple $(\Gamma_{1},\Gamma_{2},\Theta)$ 
formed by taking disjoint union. 
In the next corollary we remind the readers the standing assumptions. 
The reader is referred to \cite{freewang} for details on 
free product of compact quantum groups. 

\begin{corollary}
  \label{qisocriterion}
     Let $(\clg_{1},\clg_{2},\theta)$ and 
     $(\clg_{1}^{\prime},\clg_{2}^{\prime},\theta^{\prime})$ 
     be triples such that the vertex sets have same cardinality; 
     all the graphs $\clg_{l},\clg_{l}^{\prime}$ for $l=1,2$ are connected, 
     bidirected without loops or multiple edges. 
     Then $(\clg_{1},\clg_{2},\theta)$ is not quantum equivalent to 
     $(\clg_{1}^{\prime},\clg_{2}^{\prime},\theta^{\prime})$ 
     if and only if 
     ${\rm Aut}^{+}(\Gamma_{1},\Gamma_{2},\Theta)
     \cong 
     {\rm Aut}^{+}(\clg_{1},\clg_{2},\theta)
     \ast 
     {\rm Aut}^{+}(\clg_{1}^{\prime},\clg_{2}^{\prime},\theta^{\prime})$.   
\end{corollary}

\begin{proof}
      (\(\Rightarrow\)) Let us denote by $U$ the fundamental unitary 
      of ${\rm Aut}^{+}(\Gamma_{1},\Gamma_{2},\Theta)$ in~\eqref{Fun-unit}.
    Then by Proposition \ref{equivalent condition}, the projections 
    $p_{l}=\sum_{i\in V}q_{il}$ are zero for all $l\in V^{\prime}$. 
    By orthogonality, $q_{il}=0$ for all $i\in V,l\in V^{\prime}$. 
    Thus the upper off diagonal block becomes $0$. 
    Applying $\kappa$, the lower off diagonal block also vanishes. 
    Therefore the fundamental unitary $U$ reduces to
    \(\begin{pmatrix}
      	    Q & 0\\
            0 & Q^{\prime}
  	\end{pmatrix}\)
	where \(Q=((q_{ij}))_{i,j\in V}\) and \(Q^{\prime}=((q_{kl}))_{k,l\in V^{\prime}}\).
    Then $U$ commutes with 
    $\begin{pmatrix}
         \varepsilon_l & 0\\
         0 & \varepsilon^{\prime}_l
      \end{pmatrix}$ 
      if and only if 
      $\varepsilon_{l}Q=Q\varepsilon_{l}$ and 
      $Q^{\prime}\varepsilon_{l}^{\prime}=\varepsilon_{l}^{\prime}Q^{\prime}$. 
      As for the conditions involving $\Theta$, by the definition of $\Theta$, 
      composable pair coming from $E_{1}\star E_{2}$ is mapped to 
      $E_{2}\star E_{1}$ by $\theta$ and composable pair coming from 
      $E_{1}^{\prime}\star E_{2}^{\prime}$ is mapped to 
      $E_{2}^{\prime}\star E_{1}^{\prime}$ by $\theta^{\prime}$. 
      Again there are three cases:
  \begin{enumerate}[(i)]
      \item $\theta(e^{1}_{\alpha},e^{2}_{\beta})
      		=(e^{2}_{\mu},e^{1}_{\nu})\in E_{2}\star E_{1}$ 
		and 
		$\theta^{\prime}(e^{1}_{\alpha^{\prime}},e^{2}_{\beta^{\prime}})
		=
		(e^{2}_{\mu^{\prime}},e^{1}_{\nu^{\prime}})\in E_{2}^{\prime}\star E_{1}^{\prime}$: 
		In this case both sides of the equation \eqref{newrel} are zero.
  
       \item $\theta(e^{1}_{\alpha},e^{2}_{\beta})
       		=(e^{2}_{\mu},e^{1}_{\nu})\in E_{2}\star E_{1}$ 
		and 
		$\theta(e^{1}_{\gamma},e^{2}_{\delta})
		=
		(e^{2}_{\zeta},e^{1}_{\eta})\in E_{2}\star E_{1}$: 
		In  this case $((q_{ij}))_{i,j\in V}$ satisfies the 
		relations~\eqref{newrel} of ${\rm Aut}^{+}(\clg_{1},\clg_{2},\theta)$.
  
       \item  $\theta(e^{1}_{\alpha^{\prime}},e^{2}_{\beta^{\prime}})
       		=(e^{2}_{\mu^{\prime}},e^{1}_{\nu^{\prime}})\in E_{2}^{\prime}\star E_{1}^{\prime}$ 
		and 
		$\theta(e^{1}_{\gamma^{\prime}},e^{2}_{\delta^{\prime}})
		=
		(e^{2}_{\zeta^{\prime}},e^{1}_{\eta^{\prime}})\in E_{2}^{\prime}\star E_{1}^{\prime}$: 
		In  this case $((q_{kl}))_{k,l\in V^{\prime}}$ satisfies the 
		relations~\eqref{newrel} of ${\rm Aut}^{+}(\clg_{1}^{\prime},\clg_{2}^{\prime},\theta^{\prime})$.
  \end{enumerate} 
  Therefore, 
  ${\rm Aut}^{+}(\Gamma_{1},\Gamma_{2},\Theta)
  \cong 
  {\rm Aut}^{+}(\clg_{1},\clg_{2},\theta)
  \ast
  {\rm Aut}^{+}(\clg_{1}^{\prime},\clg_{2}^{\prime},\theta^{\prime})$.
  
  (\(\Leftarrow\)) It is easy to see that the projection $p$ from 
  Lemma~\ref{common projection} is zero and therefore by 
  Proposition~\ref{equivalent condition}, $(\clg_{1},\clg_{2},\theta)$ 
  is not quantum equivalent to $(\clg_{1}^{\prime},\clg_{2}^{\prime},\theta^{\prime})$.
\end{proof}
  
  By virtue of Corollary~\ref{qisocriterion} we provide an example 
  of a non quantum equivalent pair of triples. 
  To do this, let us consider the following pair of triples:
  \begin{enumerate}[(i)]
     \item $\clg_{1}=\clg_{2}=K_{4}$,  $\theta={\rm id}$;
     \item $\clg_{1}^{\prime}=\clg_{2}^{\prime}=K_{4}$,  $\theta^{\prime}$ 
     	     is given as in example (c) of Section \(4\)
  \end{enumerate}
  and make the disjoint union $(\Gamma_{1},\Gamma_{2},\Theta)$. 
  We are going to use the properties of free wreath product of 
  compact quantum groups by the quantum permutation groups. 
  For details on free wreath product, the reader is referred to~\cite{bichon2004free}.
  
\begin{theorem}
  \label{main computation}
    ${\rm Aut}^{+}(\Gamma_{1},\Gamma_{2},\Theta)
      \cong 
      {\rm Aut}^{+}(\clg_{1},\clg_{2},\theta)
      \ast
      {\rm Aut}^{+}(\clg_{1}^{\prime},\clg_{2}^{\prime},\theta^{\prime})
      \cong 
      S_{4}^{+}$ 
      and consequently the triples are not quantum equivalent.
\end{theorem}
  
\begin{proof}
   We label the vertices of $\clg_{1}=\clg_{2}=K_{4}$ by $1,2,3,4$ 
   and the vertices of $\clg_{1}^{\prime}=\clg_{2}^{\prime}=K_{4}$ by $5,6,7,8$ 
   as shown in the following figure:
   \begin{center}    
     \begin{tikzpicture}[scale=.5]
	\Vertex[label=$1$,position=above,shape=circle, size=0.25,color=black]{1}
	\Vertex[x=6,label=$2$,position=above,shape=circle, size=0.25,color=black]{2}
	\Vertex[y=-6,label=$3$,position=below,shape=circle, size=0.25,color=black]{3}
	\Vertex[x=6,y=-6,label=$4$,position=below,shape=circle, size=0.25,color=black]{4}

	\Vertex[x=12,label=$5$,position=above,shape=circle, size=0.25,color=black]{5}
	\Vertex[x=18,label=$6$,position=above,shape=circle, size=0.25,color=black]{6}
	\Vertex[x=12,y=-6,label=$7$,position=below,shape=circle, size=0.25,color=black]{7}
	\Vertex[x=18,y=-6,label=$8$,position=below,shape=circle, size=0.25,color=black]{8}

	\Edge[color=red,bend=20,label=$e_3$,Direct](1)(2);
	\Edge[color=red,label=$e_4$,Direct](2)(1);
	\Edge[color=red,label=$e_2$,Direct](1)(3);
	\Edge[color=red,bend=20,label=$e_1$,Direct](3)(1);
	\Edge[color=red,label=$e_8$,Direct](3)(4);
	\Edge[color=red,bend=20,label=$e_7$,Direct](4)(3);
	\Edge[color=red,bend=20,label=$e_6$,Direct](2)(4);
	\Edge[color=red,label=$e_5$,Direct](4)(2);
	\Edge[color=red,bend=20,label=$e_{12}$,Direct](4)(1);
	\Edge[color=red,bend=20,label=$e_{11}$,Direct](1)(4);
	\Edge[color=red,bend=20,label=$e_{10}$,Direct](2)(3);
	\Edge[color=red,bend=20,label=$e_9$,Direct](3)(2);

	\Edge[color=blue,bend=20,label=$f_3$,Direct](5)(6);
	\Edge[color=blue,label=$f_4$,Direct](6)(5);
	\Edge[color=blue,label=$f_2$,Direct](5)(7);
	\Edge[color=blue,bend=20,label=$f_1$,Direct](7)(5);
	\Edge[color=blue,label=$f_8$,Direct](7)(8);
	\Edge[color=blue,bend=20,label=$f_7$,Direct](8)(7);
	\Edge[color=blue,bend=20,label=$f_6$,Direct](6)(8);
	\Edge[color=blue,label=$f_5$,Direct](8)(6);
	\Edge[color=blue,bend=20,label=$f_{12}$,Direct](8)(5);
	\Edge[color=blue,bend=20,label=$f_{11}$,Direct](5)(8);
	\Edge[color=blue,bend=20,label=$f_{10}$,Direct](6)(7);
	\Edge[color=blue,bend=20,label=$f_9$,Direct](7)(6);
     \end{tikzpicture}
   \end{center}
   Recall the composable pairs from example (c) of Section \(4\). 
   We do not mention them here. We take $\theta$ to be identity on 
   \(E\star E\) for the red graph and the non trivial 
   $\theta^{\prime}$ as in the example (c) for the blue graph. 
   Recall \(\theta^{\prime}\)  is given by
      \begin{eqnarray*}
       &&\theta^{\prime}(f_{1},f_{2})=(f_{4},f_{3})\\
       &&\theta^{\prime}(f_{4},f_{3})=(f_{1},f_{2})\\
       &&\theta^{\prime}(f_{8},f_{7})=(f_{6},f_{5})\\
       &&\theta^{\prime}(f_{6},f_{5})=(f_{8},f_{7}),
   \end{eqnarray*}
   and $\theta^{\prime}$ acts as identity on the rest of the composable pairs. 
   The vertex matrices of the graphs $\Gamma_{l}$ for $l=1,2$ are given by 
   $\begin{pmatrix}
      \varepsilon & 0\\
      0 & \varepsilon
  \end{pmatrix}$ 
  where $\varepsilon$ is the vertex matrix of $K_{4}$. 
  We write the fundamental unitary $U$ of 
  ${\rm Aut}^{+}(\Gamma_{1},\Gamma_{2},\Theta)$ 
  as the following block matrix:
  \begin{equation*}
     \left( 
       \begin{BMAT}(e)[3pt,1cm,1cm]{cccc.cccc}{cccc.cccc}
            q_{11} & q_{12} & q_{13} & q_{14} & q_{15} & q_{16} & q_{17} & q_{18}\\
            q_{21} & q_{22} & q_{23} & q_{24} & q_{25} & q_{26} & q_{27} & q_{28}\\
            q_{31} & q_{32} & q_{33} & q_{34} & q_{35} & q_{36} & q_{37} & q_{38}\\
            q_{41} & q_{42} & q_{43} & q_{44} & q_{45} & q_{46} & q_{47} & q_{48}\\
            q_{51} & q_{52} & q_{53} & q_{54} & q_{55} & q_{56} & q_{57} & q_{58}\\
            q_{61} & q_{62} & q_{63} & q_{64} & q_{65} & q_{66} & q_{67} & q_{68}\\
            q_{71} & q_{72} & q_{73} & q_{74} & q_{75} & q_{76} & q_{77} & q_{78}\\
            q_{81} & q_{82} & q_{83} & q_{84} & q_{85} & q_{86} & q_{87} & q_{88} 
        \end{BMAT}
     \right)
   \end{equation*}
   Then the fundamental unitary $U$ commutes with 
   $\begin{pmatrix}
      \varepsilon & 0\\
      0 & \varepsilon
   \end{pmatrix}$ 
   if and only if the entries of $U$ satisfy the relations of 
   $S_{4}^{+}\wr_{\ast}\mathbb{Z}_{2}$. 
   Now we shall prove that the upper off diagonal block of $U$ vanishes 
   forcing the lower off diagonal block to vanish as well (using the antipode). 
   To that end, let us recall that for any composable pair $(e_{\alpha},e_{\beta})$ 
   coming from the red graph and $(e_{\alpha^{\prime}},e_{\beta^{\prime}})$ 
   coming from the blue graph such that 
   $\theta^{\prime}(e_{\alpha^{\prime}},e_{\beta^{\prime}})
   =
   (e_{\mu^{\prime}},e_{\nu^{\prime}})\in E\star E$ 
   and 
   $\theta(e_{\alpha},e_{\beta})
   =
   (e_{\alpha},e_{\beta})\in E\star E$ , 
   we have
   \begin{equation}
     \label{newrelexample}
        q_{t(e_{\alpha})t(e_{\mu^{\prime})}}
        q_{s(e_{\alpha})s(e_{\mu^{\prime}})}
        q_{s(e_{\beta})s(e_{\nu^{\prime}})}
        =
        q_{t(e_{\alpha})t(e_{\alpha^{\prime}})}
        q_{s(e_{\alpha})s(e_{\alpha^{\prime}})}
        q_{s(e_{\beta})s(e_{\beta^{\prime}})}.
   \end{equation}
   We start by showing 
   \begin{equation*}
      q_{16}=q_{17}=0.
   \end{equation*}
   Taking $(e_{\alpha^{\prime}},e_{\beta^{\prime}})=(f_{6},f_{5})$ 
   and  $(e_{\mu^{\prime}},e_{\nu^{\prime}})=(f_{8}, f_{7})$ 
   in Equation~\eqref{newrelexample}, 
   we get for any $(e_{i},e_{j})\in E\star E$,
   \begin{equation}
     \label{newnewrel2}
      q_{t(e_{i})8}q_{s(e_{i})7}q_{s(e_{j})8}
      =
      q_{t(e_{i})8}q_{s(e_{i})6}q_{s(e_{j})8}.
    \end{equation}
    Plugging in composable pairs from the red graph, we get the following relations:
    \begin{equation}
    \label{6may2:13}
    \begin{aligned}
        (e_{3},e_{4}) &: q_{28}q_{17}q_{28}=q_{28}q_{16}q_{28}\\
        (e_{3},e_{1}) &: q_{28}q_{17}q_{38}=q_{28}q_{16}q_{38}\\
        (e_{3},e_{12}) &:  q_{28}q_{17}q_{48}=q_{28}q_{16}q_{48}\\
        (e_{2},e_{4}) &:  q_{38}q_{17}q_{28}=q_{38}q_{16}q_{28}\\
        (e_{2},e_{1}) &: q_{38}q_{17}q_{38}=q_{38}q_{16}q_{38}\\
        (e_{2},e_{12}) &: q_{38}q_{17}q_{48}=q_{38}q_{16}q_{48}\\
        (e_{11},e_{4}) &: q_{48}q_{17}q_{28}=q_{48}q_{16}q_{28}\\
        (e_{11},e_{1}) &: q_{48}q_{17}q_{38}=q_{48}q_{16}q_{38}\\
        (e_{11},e_{12}) &: q_{48}q_{17}q_{48}=q_{48}q_{16}q_{48}.   
    \end{aligned}
    \end{equation}
    Clearly, for $k=1,2,\ldots,8$,
    \[ 
       q_{18}q_{16}q_{k8}=0=q_{18}q_{17}q_{k8}.
    \]
    For any $k=5,\ldots,8$,  by free wreath product condition 
    \begin{align*}
       &q_{28}q_{16}q_{k8}=0=q_{28}q_{17}q_{k8} \\
       &q_{38}q_{16}q_{k8}=0=q_{38}q_{17}q_{k8} \\
       &q_{48}q_{16}q_{k8}=0=q_{48}q_{17}q_{k8}
    \end{align*}
    By orthogonality,
    \begin{align*}
       &q_{28}q_{16}q_{18}=0=q_{28}q_{17}q_{18}\\
       &q_{38}q_{16}q_{18}=0=q_{38}q_{17}q_{18}\\
       &q_{48}q_{16}q_{18}=0=q_{48}q_{17}q_{18}
    \end{align*}
    Again by using wreath product construction, for $i=5,\ldots,8$ and $k=1,\ldots,8$, 
    \[
       q_{i8}q_{16}q_{k8}=q_{i8}q_{17}q_{k8}=0.
    \]
   Summing all we get, 
   \[
     (\sum_{i=1}^{8}q_{i8})q_{16}(\sum_{i=1}^{8}q_{i8})
     =
     (\sum_{i=1}^{8}q_{i8})q_{17}(\sum_{i=1}^{8}q_{i8}).
   \]
   Which shows $q_{16}=q_{17}$. Therefore, by orthogonality, $q_{16}=q_{17}=0$.
   Similarly, one can show that
   \begin{displaymath}
      q_{26}=q_{27}=q_{36}=q_{37}=q_{46}=q_{47}=0.
   \end{displaymath}
   Hence, for $k=1,2,3,4$ and $q_{ij}, q_{mn}$ in the upper off diagonal block, 
   \begin{displaymath}
      q_{k6}(q_{ij}-q_{mn})=0.
   \end{displaymath}
   For any $k=5,6,7,8$, by the free wreath product relations,
   \begin{displaymath}
      q_{k6}(q_{ij}-q_{mn})=0.
   \end{displaymath}
   Therefore, 
   $\sum_{k=1}^{8}q_{k6}(q_{ij}-q_{mn})=0$ which shows $q_{ij}=q_{mn}$. 
   In particular, by orthogonality, 
   $q_{15}=q_{18}=q_{25}=q_{28}=q_{35}=q_{38}=q_{45}=q_{48}=0$. 
   Consequently, the upper off diagonal block becomes a zero matrix. 
   Applying the antipode $\kappa$, the lower off diagonal also becomes a zero matrix. 
   Therefore using the same argument as used in the proof of Corollary~\ref{qisocriterion} , 
   we have
   \begin{displaymath}
       {\rm Aut}^{+}(\Gamma_{1},\Gamma_{2},\Theta)
       \cong 
       {\rm Aut}^{+}(\clg_{1},\clg_{2},\theta)
       \ast
       {\rm Aut}^{+}(\clg_{1}^{\prime},\clg_{2}^{\prime},\theta^{\prime})
       \cong 
       S_{4}^{+}.
   \end{displaymath}
    As the triples satisfy all the standing assumptions, 
    by Corollary~\ref{qisocriterion}, it follows that
    $(K_{4},K_{4},{\rm id})\nsim^{q} (K_{4},K_{4},\theta^{\prime})$
\end{proof}

 \begin{remark}
     The above example is non trivial in the sense that we have taken 
     $\clg_{l}=\clg_{l}^{\prime}=K_{4}$ and therefore all of them are trivially 
     (quantum) isomorphic. Only the choice of different bijections between 
     the composable pairs make the triples non quantum equivalent.
 \end{remark}
 
Now we are naturally led to take the same $\theta$ and compute the quantum 
automorphism group of disjoint union of the same copy of a  triple 
$(\clg_{1},\clg_{2},\theta)$. As such there is is no need to restrict 
ourselves to only two copies. To accommodate taking disjoint union of $n$-many 
copies of a triple, we introduce some notations.

 Let $\clg_{1},\clg_{2}$ be two finite, connected bidirected graphs 
 with commuting vertex matrices on the same vertex set $V$. 
 Denoting the edge sets by $E_{1}, E_{2}$ respectively, let $\theta$ 
 be the bijection between $E_{1}\star E_{2}$ and $E_{2}\star E_{1}$ 
 so that $(\clg_{1},\clg_{2},\theta)$ becomes a defining triple for a $2$-graph. 
 Then we construct the triple to be denoted by 
 $(\clg_{1}^{(n)},\clg_{2}^{(n)},\theta^{(n)})$ where $\clg_{l}^{(n)}$ 
 are taken to be the disjoint union of $n$-copies of the graphs $\clg_{l}$ 
 on the vertex set $\clv=\underbrace{V\sqcup V\sqcup\ldots\sqcup V}_{n-times}$ for $l=1,2$. 
 Then the composable pairs become 
 $\underbrace{(E_{1}\star E_{2})\sqcup\ldots\sqcup(E_{1}\star E_{2})}_{n-times}$. 
 The bijection $\theta^{(n)}$ is defined to be $\theta$ on each copy 
 of the disjoint union of composable pairs. We also denote the $k$-th 
 copy of the disjoint union of composable pairs by $E_{1(k)}\star E_{2(k)}$ 
 and the corresponding bijection by $\theta_{(k)}$. 
 We do not label the individual edges of the disjoint union as it would 
 make notations cumbersome. Instead, we shall explicitly mention from 
 which copy the edges are coming when needed. 
 For a CQG \(\mathbb{G}\), \(\mathbb{G}^{*n}\) will denote the \(n\)-fold 
 free product of \(\mathbb{G}\) with itself. With these notations we 
 have the following analogue of  \cite{bichon2004free}*{Theorem 4.2}.

\begin{theorem}    
    \begin{displaymath}
        {\rm Aut}^{+}(\clg_{1}^{(n)},\clg_{2}^{(n)},\theta^{(n)})
        =
        {\rm Aut}^{+}(\clg_{1},\clg_{2},\theta) \wr_{\ast} S_{n}^{+}.
    \end{displaymath} 
\end{theorem}

\begin{proof}
    Let $\clg_{1},\clg_{2}$ be finite connected bidirected graphs 
    with set of vertices $V=\{1, \ldots, m\}$. For $k \in\{1, \ldots, n\}$, 
    we denote by $\clg_{1(k)},\clg_{2(k)}$ the $k$-th copy of 
    $\clg_{1},\clg_{2}$ in $\clg_{1}^{(n)},\clg_{2}^{(n)}$ respectively, 
    with set of vertices $V_k=\{k(1), \ldots, k(m)\}$. 
    The generators of ${\rm Aut}^{+}(\clg_{1},\clg_{2},\theta)$ 
    are denoted by $\left(u_{i j}\right)_{1 \leq i, j \leq m}$, 
    and the generators of $S_{n}^{+}$ are  denoted by 
    $\left(x_{k l}\right)_{1 \leq k, l \leq n}$. Let $\mathcal{B}$ 
    be the free algebra with generators $X_{k(i) l(j)}, 1 \leq i, j \leq m$, 
    $1 \leq k, l \leq n$; let $((u_{ij}))_{i,j\in V}$ be the fundamental unitary 
    of ${\rm Aut}^{+}(\clg_{1},\clg_{2},\theta)$. 
    We define an algebra morphism
    \begin{align*}
	\Phi_0\colon &\mathcal{B} 
	\longrightarrow 
	{\rm Aut}^{+}(\clg_{1},\clg_{2},\theta)\wr_{\ast} S_{n}^{+}\\
	 &X_{k(i) l(j)} \longmapsto \nu_k\left(u_{i j}\right) x_{k l},
    \end{align*} 
    where 
    $\nu_{k}:{\rm Aut}^{+}(\clg_{1},\clg_{2},\theta)
    \rightarrow 
    {\rm Aut}^{+}(\clg_{1},\clg_{2},\theta)^{\ast n}$ 
    is the embedding of the $k$-th copy in the free product.
    Now we can easily verify that it factors through all the relations for 
    ${\rm Aut}^{+}(\clg_{1}^{(n)},\clg_{2}^{(n)},\theta^{(n)})$. 
    To that end, first we arrange $\Phi_{0}(X_{k(i)l(j)})$ in a block matrix 
    indexed by the set $\clv\times\clv$. We do this by labeling the $i$-th 
    vertex of the $k$-th copy of $V$ by $k(i)$. If the vertex matrices of 
    $\clg_{1},\clg_{2}$ are $\varepsilon_{1},\varepsilon_{2}$, then the vertex matrices for 
    $\clg_{1}^{(n)}, \clg_{2}^{(n)}$ are block diagonal matrices with the diagonal blocks being 
    $\varepsilon_{1}, \varepsilon_{2}$. Then it follows verbatim from the proof 
    of \cite{bichon2004free}*{Theorem 4.2} that $((\Phi_{0}(X_{k(i)l(j)})))$ 
    commutes with the vertex matrices of $\clg_{1}^{(n)}, \clg_{2}^{(n)} $. 
    As for the relations~\eqref{newrel}, by the definition of $\theta^{(n)}$ 
    we have to consider the following two cases: 

    \subsection*{Case I}
     Let $\theta^{(n)}(e_{\alpha},e_{\beta})=(e_{\mu},e_{\nu})$  and 
     $\theta^{(n)}(e_{\alpha^{\prime}},e_{\beta^{\prime}})
     =(e_{\mu^{\prime}},e_{\nu^{\prime}})$ 
     where 
     $(e_{\alpha},e_{\beta})$, $(e_{\alpha^{\prime}},e_{\beta^{\prime}}) \in E_{1(k)}\star E_{2(k)}$ 
     and 
     $(e_{\mu},e_{\nu}),(e_{\mu^{\prime}},e_{\nu^{\prime}}) \in E_{2(k)}\star E_{1(k)}$.
     Let 
     \begin{equation*}
       \begin{matrix}
        t(e_{\mu^{\prime}})=k(i^{\prime}), 
        &
        t(e_{\mu})=k(i), 
        &
        s(e_{\mu^{\prime}})=k(j^{\prime}), 
        &
        s(e_{\mu})=k(j), \\
        s(e_{\nu^{\prime}})=k(p^{\prime}), 
        &
        s(e_{\nu})=k(p), 
        &
        t(e_{\alpha^{\prime}})=k(q^{\prime}), 
        & 
        t(e_{\alpha})=k(q),\\
 	s(e_{\alpha^{\prime}})=k(r^{\prime}),
	&
	s(e_{\alpha})=k(r), 
	&
	s(e_{\beta^{\prime}})=k(t^{\prime}), 
	&
	s(e_{\beta})=k(t).
      \end{matrix}
    \end{equation*}
    Then, we calculate
      \begin{align*}
        \Phi_{0}(
        		&X_{t(e_{\mu^{\prime}})t(e_{\mu})} 
		X_{s(e_{\mu^{\prime}})s(e_{\mu})} 
		X_{s(e_{\nu^{\prime}})s(e_{\nu})}) \\
		&= 
		\Phi_{0}(
		X_{k(i^{\prime})k(i)} 
		X_{k(j^{\prime})k(j)} 
		X_{k(p^{\prime})k(p)}) \\
        		&= 
		\left(\nu_k\left(u_{i^{\prime}i}\right) x_{k k}\right) 
		\left( \nu_k\left(u_{j^{\prime}j}\right) x_{k k}\right) 
		\left( \nu_k\left(u_{p^{\prime} p}\right) x_{k k}\right) \\
        		&=
		 \bigl(\nu_{k}(u_{i^{\prime}i}u_{j^{\prime}j}u_{p^{\prime}p})\bigr)
		 x_{kk}x_{kk}x_{kk}.
    \end{align*}
   
    \begin{align*}
        \Phi_{0}(
        	   &X_{t(e_{\alpha^{\prime}})t(e_{\alpha})} 
	   X_{s(e_{\alpha^{\prime}})s(e_{\alpha})} 
	   X_{s(e_{\beta^{\prime}})s(e_{\beta})}) \\
	   &= 
	   \Phi_{0}(
	   X_{k(q^{\prime})k(q)} 
	   X_{k(r^{\prime})k(r)} 
	   X_{k(t^{\prime})k(t)}) \\
        	  &= 
	  \left(\nu_k\left(u_{q^{\prime}q}\right) x_{k k}\right) 
	  \left( \nu_k\left(u_{r^{\prime} r}\right) x_{k k}\right) 
	  \left( \nu_k\left(u_{t^{\prime} t}\right) x_{k k}\right) \\
          &= 
          \bigl(\nu_{k}(u_{q^{\prime}q}u_{r^{\prime}r}u_{t^{\prime} t})\bigr)
          x_{kk}x_{kk}x_{kk}.
    \end{align*}
    Since $\theta(e_{\alpha},e_{\beta})=(e_{\mu},e_{\nu})$, 
    $\theta(e_{\alpha^{\prime}},e_{\beta^{\prime}})
    =(e_{\mu^{\prime}},e_{\nu^{\prime}})$ 
    and $u_{ij}$ is the fundamental representation for 
    $Aut^+(\clg_{1},\clg_{2},\theta)$, we have
    \begin{equation}
     \label{uij-compat}
       	u_{t(e_{\mu^{\prime}})t(e_{\mu})} 
    	u_{s(e_{\mu^{\prime}})s(e_{\mu})} 
    	u_{s(e_{\nu^{\prime}})s(e_{\nu})}
   	 =
    	u_{t(e_{\alpha^{\prime}})t(e_{\alpha})} 
    	u_{s(e_{\alpha^{\prime}})s(e_{\alpha})}
    	u_{s(e_{\beta^{\prime}})s(e_{\beta})}
    \end{equation} 
    and therefore 
    $u_{i^{\prime}i}u_{j^{\prime}j}u_{p^{\prime}p}
    =u_{q^{\prime}q}u_{r^{\prime}r}u_{t^{\prime} t}$. 
    So, in this case,
    \begin{displaymath}
  	\Phi_{0}(X_{t(e_{\mu^{\prime}})t(e_{\mu})} 
  	X_{s(e_{\mu^{\prime}})s(e_{\mu})} 
  	X_{s(e_{\nu^{\prime}})s(e_{\nu})})
  	= 
  	\Phi_{0}(X_{t(e_{\alpha^{\prime}})t(e_{\alpha})} 
  	X_{s(e_{\alpha^{\prime}})s(e_{\alpha})} 
  	X_{s(e_{\beta^{\prime}})s(e_{\beta})}). 
    \end{displaymath}
    
    \subsection*{Case II} 
    Let $\theta^{(n)}(e_{\alpha},e_{\beta})=(e_{\mu},e_{\nu})$ 
    and $\theta^{(n)}(e_{\alpha^{\prime}},e_{\beta^{\prime}})
    =
    (e_{\mu^{\prime}},e_{\nu^{\prime}})$ 
    such that 
    $(e_{\alpha},e_{\beta})\in E_{1(k)}\star E_{2(k)}$, 
    $(e_{\alpha^{\prime}},e_{\beta^{\prime}}) \in E_{1(l)}\star E_{2(l)}$, 
    $(e_{\mu},e_{\nu})\in E_{2(k)}\star E_{1(k)}$ 
    and 
    $(e_{\mu^{\prime}},e_{\nu^{\prime}}) \in E_{2(l)}\star E_{1(l)}$ 
    where $k\ne l\in\{1,2,\cdots,n\}$. 
    Let
    \begin{equation*}
      \begin{matrix}
        t(e_{\mu^{\prime}})=l(i^{\prime}), 
        & 
        t(e_{\mu})=k(i), 
        & 
        s(e_{\mu^{\prime}})=l(j^{\prime}), 
        &
        s(e_{\mu})=k(j),\\
	s(e_{\nu^{\prime}})=l(p^{\prime}), 
	&
	s(e_{\nu})=k(p), 
	&
	t(e_{\alpha^{\prime}})=l(q^{\prime}), 
	&
	t(e_{\alpha})=k(q),\\
	s(e_{\alpha^{\prime}})=l(r^{\prime}), 
	&
	s(e_{\alpha})=k(r), 
	&
	s(e_{\beta^{\prime}})=l(t^{\prime}), 
	&
	s(e_{\beta})=k(t).
      \end{matrix}
    \end{equation*}
    And again we calculate       
    \begin{align*}
        \Phi_{0}(
        &X_{t(e_{\mu^{\prime}})t(e_{\mu})} 
        X_{s(e_{\mu^{\prime}})s(e_{\mu})} 
        X_{s(e_{\nu^{\prime}})s(e_{\nu})}) \\
        &= 
        \Phi_{0}(
        X_{l(i^{\prime})k(i)} 
        X_{l(j^{\prime})k(j)} 
        X_{l(p^{\prime})k(p)}) \\
        &= \left(\nu_l\left(u_{i^{\prime} i}\right) 
        x_{l k}\right) \left( \nu_l\left(u_{j^{\prime} j}\right) 
        x_{l k}\right) \left( \nu_l\left(u_{p^{\prime}p}\right) 
        x_{l k}\right) \\
        &= \left(\nu_{l}(u_{i^{\prime}i}u_{j^{\prime}j}u_{p^{\prime}p})\right)
        x_{lk}x_{lk}x_{lk}.  
    \end{align*}
    \begin{align*}
        \Phi_{0}(
        &X_{t(e_{\alpha^{\prime}})t(e_{\alpha})} 
        X_{s(e_{\alpha^{\prime}})s(e_{\alpha})} 
        X_{s(e_{\beta^{\prime}})s(e_{\beta})})\\
        &= 
        \Phi_{0}(
        X_{l(q^{\prime})k(q)} 
        X_{l(r^{\prime})k(r)} 
        X_{l(t^{\prime})k(t)}) \\
        &= \left(\nu_l\left(u_{q^{\prime} q}\right) 
        x_{l k}\right) \left( \nu_l\left(u_{r^{\prime} r}\right) 
        x_{l k}\right) \left( \nu_l\left(u_{t^{\prime} t}\right) 
        x_{l k}\right) \\
        &= \left(\nu_{l}(u_{q^{\prime}q}u_{r^{\prime}r}u_{t^{\prime} t})\right)
        x_{lk}x_{lk}x_{lk}.  
    \end{align*}
    It follows from equation~\eqref{uij-compat} that
    $u_{i^{\prime}i}u_{j^{\prime}j}u_{p^{\prime}p}
    =u_{q^{\prime}q}u_{r^{\prime}r}u_{t^{\prime} t}$. 
    Hence
    \begin{displaymath}
      \Phi_{0}(
      X_{t(e_{\mu^{\prime}})t(e_{\mu})} 
      X_{s(e_{\mu^{\prime}})s(e_{\mu})} 
      X_{s(e_{\nu^{\prime}})s(e_{\nu})})
      = 
      \Phi_{0}(
      X_{t(e_{\alpha^{\prime}})t(e_{\alpha})} 
      X_{s(e_{\alpha^{\prime}})s(e_{\alpha})} 
      X_{s(e_{\beta^{\prime}})s(e_{\beta})}). 
    \end{displaymath}
    The above calculations show that there is a well-defined $\ast$-homomorphism
    \begin{displaymath}
        \Phi:{\rm Aut}^{+}(\clg_{1}^{(n)},\clg_{2}^{(n)},\theta^{(n)})
        \longrightarrow 
        {\rm Aut}^{+}(\clg_{1},\clg_{2},\theta)\wr_{\ast}S^{+}_{n}.
    \end{displaymath}
    In order to construct the inverse of $\Phi$, 
    let $\mathcal{C}$ be the free $\ast$-algebra with generators $(u_{ij})_{1\leq i,j\leq m}$. 
    For $k\in \{1,2,\cdots,n\}$, define $\ast$-algebra homomorphism 
    \begin{displaymath}
  	\eta_{0}^{k}: \mathcal{C}
  	\longrightarrow 
  	{\rm Aut}^{+}(\clg_{1}^{(n)},\clg_{2}^{(n)},\theta^{(n)}), 
  	\quad 
  	u_{ij}\mapsto \sum_{l=1}^{n}X_{k(i)l(j)}, 
  	\quad 1\leq i,j \leq m.
    \end{displaymath}
    We show that this factors through the ideal generated by the relations of 
    ${\rm Aut}^{+}(\clg_{1},\clg_{2},\theta)$. 
    Again as before, the only new thing is to take care of is the 
    new relations~\eqref{newrel}. To that end, let 
    $\theta(e_{\alpha},e_{\beta})=(e_{\mu},e_{\nu})$ 
    and 
    $\theta(e_{\alpha^{\prime}},e_{\beta^{\prime}})=(e_{\mu^{\prime}},e_{\nu^{\prime}})$. 
    We shall show that for any $k\in \{1,2,\cdots,n\}$,
    \begin{align*}
 	\eta_{0}^{k}(&u_{t(e_{\mu^{\prime}})t(e_{\mu})})
 	\eta^{k}_{0} (u_{s(e_{\mu^{\prime}})s(e_{\mu})}) 
 	\eta^{k}_{0}(u_{s(e_{\nu^{\prime}})s(e_{\nu})})\\
 	&=
	\eta_{0}^{k}(u_{t(e_{\alpha^{\prime}})t(e_{\alpha})})
 	\eta^{k}_{0}( u_{s(e_{\alpha^{\prime}})s(e_{\alpha})})
 	\eta^{k}_{0}( u_{s(e_{\beta^{\prime}})s(e_{\beta})}).
    \end{align*}
    So let 
    \begin{equation*}
      \begin{matrix}
     	t(e_{\mu^{\prime}})=i^{\prime}, 
	&
	t(e_{\mu})=i, 
	&
	s(e_{\mu^{\prime}})=j^{\prime}, 
	&
	s(e_{\mu})=j,\\
     	s(e_{\nu^{\prime}})=p^{\prime}, 
	&
	s(e_{\nu})=p, 
	&
	t(e_{\alpha^{\prime}})=q^{\prime}, 
	&
	t(e_{\alpha})=q,\\
     	s(e_{\alpha^{\prime}})=r^{\prime}, 
	&
	s(e_{\alpha})=r, 
	&
	s(e_{\beta^{\prime}})=t^{\prime}, 
	&
	s(e_{\beta})=t.
      \end{matrix}
    \end{equation*} 
    Thus we have to show that 
    \begin{equation*}
  	\eta_{0}^{k}(u_{i^{\prime} i})
  	\eta_{0}^{k} (u_{j^{\prime} j})
  	\eta^{k}_{0} (u_{p^{\prime} p})
  	=
  	\eta_{0}^{k}(u_{q^{\prime} q})
  	\eta^{k}_{0} (u_{r^{\prime} r})
  	\eta^{k}_{0} (u_{t^{\prime} t}),    
    \end{equation*}
    i.e.
    \begin{equation}
      \label{09may9:54}
   	\sum_{l,l^{\prime},l^{\prime\prime}}^{n}
   	X_{k(i^{\prime})l(i)}X_{k(j^{\prime})l^{\prime}(j)}
   	X_{k(p^{\prime})l^{\prime\prime}(p)} 
   	= 
   	\sum_{l,l^{\prime},l^{\prime\prime}}^{n}
   	X_{k(q^{\prime})l(q)}
   	X_{k(r^{\prime})l^{\prime}(r)}
   	X_{k(t^{\prime})l^{\prime\prime}(t)}.   
    \end{equation}
    Consider the composable pairs $(e_{\alpha},e_{\beta})$ and 
    $(e_{\alpha^{\prime}},e_{\beta^{\prime}})$ from the $l$-th and $k$-th copy 
    of the disjoint union of composable pairs and take  
    $\theta_{(l)}(e_{\alpha},e_{\beta})=(e_{\mu},e_{\nu})$ and 
    $\theta_{(k)}(e_{\alpha^{\prime}},e_{\beta^{\prime}})
    =(e_{\mu^{\prime}},e_{\nu^{\prime}})$ 
    so that we can do the following indexing:
    \begin{equation*}
      \begin{matrix}
      	t(e_{\mu^{\prime}})=k(i^{\prime}), 
	&
	t(e_{\mu})=l(i), 
	&
	s(e_{\mu^{\prime}})=k(j^{\prime}), 
	&
	s(e_{\mu})=l(j),\\
      	s(e_{\nu^{\prime}})=k(p^{\prime}), 
	&
	s(e_{\nu})=l(p), 
	&
	t(e_{\alpha^{\prime}})=k(q^{\prime}), 
	&
	t(e_{\alpha})=l(q),\\
      	s(e_{\alpha^{\prime}})=k(r^{\prime}), 
	&
	s(e_{\alpha})=l(r), 
	&
	s(e_{\beta^{\prime}})=k(t^{\prime}), 
	&
	s(e_{\beta})=l(t).
      \end{matrix}
    \end{equation*} 
    Since $k(i^{\prime})=t(e_{\mu^{\prime}})$ and $k(j^{\prime})=s(e_{\mu^{\prime}})$ 
    and there is no edge between the vertices $l(i)$ and $l^{\prime}(j)$ for 
    $l\neq l^{\prime}$, $X_{k(i^{\prime})l(i)}X_{k(j^{\prime})l^{\prime}(j)}=0$. 
    Similarly, $X_{k(q^{\prime})l(q)}X_{k(r^{\prime})l^{\prime}(r)}=0$ for $l\neq l^{\prime}$. 
    Therefore, \eqref{09may9:54} reduces to  
    \begin{equation}
      \label{wr1}
     	\sum_{l,l^{\prime\prime}}^{n}
      	X_{k(i^{\prime})l(i)}
      	X_{k(j^{\prime})l(j)}
      	X_{k(p^{\prime})l^{\prime\prime}(p)} 
      	= 
      	\sum_{l,l^{\prime\prime}}^{n}
      	X_{k(q^{\prime})l(q)}
      	X_{k(r^{\prime})l(r)}
      	X_{k(t^{\prime})l^{\prime\prime}(t)}.
    \end{equation}
    Now $k(j^{\prime})=s(e_{\mu^{\prime}})=t(e_{\nu^{\prime}})$ and  
    $k(p^{\prime})=s(e_{\nu^{\prime}})$ which shows that there is an edge between 
    $k(j^{\prime})$ and $k(p^{\prime})$. Similarly 
    $k(r^{\prime})=s(e_{\alpha^{\prime}})=t(e_{\beta^{\prime}})$ and 
    $k(t^{\prime})=s(e_{\beta^{\prime}})$. Hence there is an edge between 
    $k(r^{\prime})$ and $k(t^{\prime})$. Therefore, using the property of generators, 
    the L.H.S. of Equation~\eqref{wr1} becomes 
     \( \sum_{l=1}^{n}X_{k(i^{\prime})l(i)}X_{k(j^{\prime})l(j)}X_{k(p^{\prime})l(p)}\)
    and the R.H.S. becomes
    \(\sum_{l=1}^{n}X_{k(q^{\prime})l(q)}X_{k(r^{\prime})l(r)}X_{k(t^{\prime})l(t)}\).
    Now $\theta_{(k)}(e_{\alpha},e_{\beta})=(e_{\mu},e_{\nu})$ and 
    $\theta_{(l)}(e_{\alpha^{\prime}},e_{\beta^{\prime}})=(e_{\mu^{\prime}},e_{\nu^{\prime}})$ imply,
    \begin{displaymath}
    	X_{k(i^{\prime})l(i)}
    	X_{k(j^{\prime})l(j)}
    	X_{k(p^{\prime})l(p)}
    	=
    	X_{k(q^{\prime})l(q)}
    	X_{k(r^{\prime})l(r)}
    	X_{k(t^{\prime})l(t)}.
    \end{displaymath} 
    Therefore, we finally get 
    \begin{displaymath}
    	\eta_{0}^{k}(
    	u_{i^{\prime} i} 
    	u_{j^{\prime} j} 
    	u_{p^{\prime} p})
    	=
    	\eta_{0}^{k}(
    	u_{q^{\prime} q} 
    	u_{r^{\prime} r} 
    	u_{t^{\prime} t}).
    \end{displaymath}
     Hence it induces a *-homomorphism 
     \begin{displaymath}
 	\eta^{k}: {\rm Aut}^{+}(\clg_{1},\clg_{2},\theta) 
 	\longrightarrow 
 	{\rm Aut}^{+}(\clg_{1}^{(n)},\clg_{2}^{(n)},\theta^{(n)})
 	\quad
 	\text{ for all } k=1,...,n.
    \end{displaymath} 
    We also have *-homomorphism  
    $\pi: S_{n}^{+}\longrightarrow {\rm Aut}^{+}(\clg_{1}^{(n)},\clg_{2}^{(n)},\theta^{(n)})$ 
    such that for any $k,l\in\{1,,2,\cdots,n\}$ (see~\cite{bichon2004free}),
    \begin{displaymath}
    	\pi(x_{kl})=\sum_{r=1}^m X_{k(r)l(i)} 
    	\quad 
    	\text{ for all }  i\in \{1,2,\cdots,m\}.
    \end{displaymath} 
    Using the universal property of the free product we get *-homomorphism 
    $\Psi_{0}: {\rm Aut}^{+}(\clg_{1},\clg_{2},\theta)^{*n}* S_{n}^{+}
    \longrightarrow 
    {\rm Aut}^{+}(\clg_{1}^{(n)},\clg_{2}^{(n)},\theta^{(n)})$ 
    such that $\Psi_{0}\circ \nu_{k}=\eta^k$, where 
    $\nu_{k}:{\rm Aut}^{+}(\clg_{1},\clg_{2},\theta)
    \rightarrow 
    {\rm Aut}^{+}(\clg_{1},\clg_{2},\theta)^{*n}* S_{n}^{+}$ is the natural embedding. 
    Also for $k,l\in \{1,2,\cdots,n\}$ and $i,j\in \{1,2,\cdots,m\}$, 
    we have along the lines of \cite{bichon2004free}, 
    \begin{displaymath}
     	 \Psi_{0}(\nu_{k}(u_{ij})x_{kl})=\Psi_{0}(x_{kl}\nu_{k}(u_{ij})).
    \end{displaymath}
    Therefore, finally by the universal property of the free wreath product, 
    we get a $\ast$-homomorphism 
    $\Psi: {\rm Aut}^{+}(\clg_{1},\clg_{2},\theta)\wr_{\ast} S^{+}_{n}
    \longrightarrow 
    {\rm Aut}^{+}(\clg_{1}^{(n)},\clg_{2}^{(n)},\theta^{(n)})$. 
    Now it is easy to check that $\Phi$ and $\Psi$ are inverses of each other, 
    completing the proof. 
\end{proof}

\section*{Concluding remarks}
    \begin{itemize}
	\item The $2$-graph in the example (b) of Section~\(4\) can be seen to be 
		isomorphic to the cartesian product of the $1$-graph 
		with two vertices and a bidirected edge (say $\mathcal{G}$) 
		with itself in the sense of Proposition 1.8 of \cite{NYJM}. 
		In that example, the quantum automorphism group of the $2$-graph 
		is ${\rm Aut}^{+}(\mathcal{G})\ot{\rm Aut}^{+}(\mathcal{G})$. 
		We conjecture that if $\mathcal{G}_{1},\mathcal{G}_{2}$ are 
		two $1$-graphs, then the quantum automorphism group of the 
		$2$-graph $\mathcal{G}_{1}\times\mathcal{G}_{2}$ at least contains 
		${\rm Aut}^{+}(\mathcal{G}_{1})\ot{\rm Aut}^{+}(\mathcal{G}_{2})$ 
		as a quantum subgroup.
	\item One would like to find a `non-trivial' example of a pair of $2$-graphs 
		(or equivalently triples) that are not classically isomorphic, 
		but quantum isomorphic. We are working in this direction and we hope 
		to report it in a forthcoming article soon. To produce a ``trivial" example, 
		one can simply take a pair of quantum isomorphic graphs which are not 
		classicaly isomorphic say $\clg_{1},\clg_{2}$. 
		Then the triples $(\clg_{1}, \clg_{1}, {\rm id})$ and $(\clg_{2}, \clg_{2}, {\rm id})$ 
		will be quantum equivalent but not classicaly equivalent.
	\item The quantum automorphism group of a $2$-graph should naturally 
		act on the corresponding higher rank graph $C^{\ast}$-algebra as it 
		has already been observed in the literature in the case 
		of $1$-graphs (see \cites{Joardar1,Joardar2,Weber} for example).
	\item We believe that the notion of quantum isomorphism between $2$-graphs 
		given in the Definition \ref{qiso} can be linked to a perfect quantum strategy 
		for winning a suitable non-local ``$2$-graph isomorphism game''. 
    \end{itemize}

\appendix

\section{}

\begin{theorem}
  \label{appendix}
    Two defining triples 
    $(\mathcal{G}_1,\mathcal{H}_1,\theta_1)$ 
    and 
    $(\mathcal{G}_2,\mathcal{H}_2, \theta_2)$ are equivalent 
    if and only if the corresponding 2-graphs 
    (say $\Lambda_{1}$ and $\Lambda_2$ respectively) are isomorphic.
\end{theorem}

\begin{proof}
     This theorem must be well-known and the proof being very natural, 
     is only sketched here. Suppose~\((\mathcal{G}_1,\mathcal{H}_1,\theta_1)\) 
     and~\((\mathcal{G}_2,\mathcal{H}_2, \theta_2)\) are equivalent. 
     Therefore we may essentially assume that both 
     $(\mathcal{G}_{l},\mathcal{H}_{l},\theta_{l})$ for $l=1,2$ 
     share the common vertex set $V$; a bijection 
     $T:V\raro V$; $\mathcal{G}_{l}, \mathcal{H}_{l}$ 
     have adjacency matrices $A_{l}, B_{l}$ respectively for $l=1,2$ 
     such that
    \begin{displaymath}
        TA_{2}=A_{1}T; \ TB_{2}=B_{1}T.
    \end{displaymath}
    Let's construct a functor $\Phi$ from $\Lambda_{1}$ to $\Lambda_{2}$. 
    We recall that ${\rm obj}(\Lambda_l)$ are identified with the common vertex set $V$. 
    We assign to each $i\in V$ the element $T(i)\in V$. 
    So that we associate each object of $\Lambda_1$ to an object in $\Lambda_2$. 
    Now we want to associate each morphism in $\Lambda_{1}^{(m,n)}$ to a 
    morphism in $\Lambda_{2}^{(m,n)}$ for all $(m,n)\in\mathbb{N}^2$. 
    For $(m,n)=(1,0)$, a morphism $f$ from an object $i$ to an object $j$ 
    in $\Lambda_{1}$ is an edge (say $f$) in the graph $\mathcal{G}_{1}$ 
    with $s(f)=i$ and $t(f)=j$. Since $TA_{2}=A_{1}T$, 
    there is an edge say $\Phi(f)$  with $s(\Phi(f))=T(i)$ and $t(\Phi(f))=T(j)$. 
    Similarly one gets an edge $\Phi(g)$  with $s(\Phi(g))=T(i)$ and $t(\Phi(g))=T(j)$ 
    for a morphism $g$ of degree $(0,1)$ with $s(g)=i$ and $t(g)=j$. 
    Now let $F$ be a morphism of degree $(1,1)$ with $s(F)=i$ and $t(F)=j$. 
    By the unique factorization property, there are unique morphisms $f_{1},f_{2}$ 
    of degree $(1,0)$ and $g_{1},g_{2}$ of degree $(0,1)$ such that 
    $F=f_{1}\circ g_{1}=g_{2}\circ f_{2}$. It is clear that two morphisms 
    $f,g$ of degree $(1,0)$ or $(0,1)$ are composable iff 
    $\Phi(f)$ and $\Phi(g)$ are composable. Then we define $\Phi(F)$ 
    to be either of $\Phi(f_{1})\circ\Phi(g_{1})$ or $\Phi(g_2)\circ\Phi(f_2)$. 
    Using the $\theta$-compatibility and the fact that 
    $\theta_{1}(f_{1},g_{1})=(g_{2},f_{2})$ we get that 
    $\theta_{2}(\Phi(f_1),\Phi(g_{1}))=(\Phi(g_{2}),\Phi(f_{2}))$. 
    This proves the well-definedness of $\Phi(F)$. 
    Now, for any $(m,n)\in\mathbb{N}^{2}$, step by step one can 
    define $\Phi(F)\in \Lambda_{2}^{(m,n)}$ for $F\in\Lambda_{1}^{(m,n)}$. 
    One can easily verify that $\Phi$ is a functor. 
    Similarly, one can define a functor $\Phi^{-1}$ from $\Lambda_{2}$ to $\Lambda_{1}$ 
    using the map $T^{-1}$ on the vertex set $V$ (i.e. on the level of objects) 
    and it can be easily verified that $\Phi^{-1}$ is the inverse of the functor $\Phi$. 
    Thus an isomorphism between the categories 
    $\Lambda_1$ and $\Lambda_2$ is established.
    
    For the converse direction, we start with two isomorphic $2$-graphs 
    $\Lambda_1$ and $\Lambda_2$ with defining triples 
    $(\mathcal{G}_{1},\mathcal{H}_{1},\theta_{1})$ 
    and 
    $(\mathcal{G}_{2},\mathcal{H}_{2},\theta_{2})$ respectively. 
    Using the isomorphism of $\Lambda_1\) and \( \Lambda_2$, 
    one can see that there is a bijection between ${\rm obj}(\Lambda_1)$ 
    and ${\rm obj}(\Lambda_2)$ which are finite sets by our assumption. 
    So, without loss of generality, one can assume that $\mathcal{G}_{l},\mathcal{H}_{l}$ 
    share the same vertex set $V$. We denote the functor from 
    $\Lambda_{1}$ to $\Lambda_{2}$ by $\Phi$. 
    Then there is a morphism of degree $(1,0)$ from object $v_{i}$ to $v_{j}$ 
    i.e. there is an  edge of the graph $\mathcal{G}_{1}$ say $f$ 
    from the vertex $v_{i}$ to the vertex $v_{j}$ $\Rightarrow$ 
    there is a unique morphism $\Phi(f)$ from $\Phi(v_{i})$ to $\Phi(v_{j})$ 
    i.e. there is an edge from the vertex $\Phi(v_{i})$ to the vertex $\Phi(v_{j})$. 
    Equivalently if the permutation matrix corresponding to the 
    bijection $\Phi$ is denoted by $T$, $TA_{2}=A_{1}T$ where $A_{l}$ are 
    the vertex matrices of $\mathcal{G}_{l}$ for $l=1,2$. 
    Similarly considering degree $(0,1)$ morphisms one gets $TB_{2}=B_{1}T$. 
    Let $F$ be a morphism of degree $(1,1)$ in $\Lambda_1$ 
    with unique factorization $f_{1}\circ g_{1}=g_{1}^{\prime}\circ f_{1}^{\prime}$. 
    Then, as usual, identifying morphisms of degree $(1,0)$ and $(0,1)$ in 
    $\Lambda_{1}$ with edges of graphs $\mathcal{G}_{1}, \mathcal{H}_{1}$ respectively, 
    we get $\theta_{1}(f_{1},g_{1})=(g_{1}^{\prime},f_{1}^{\prime})$. 
    But by the definition of functor, 
    $\Phi(F)=\Phi(f_1)\circ\Phi(g_1)=\Phi(g_{1}^{\prime})\circ\Phi(f_{1}^{\prime})$. 
    Therefore from the unique factorization of $\Phi(F)$ in $\Lambda_2$, 
    we obtain that 
    $\theta_2(\Phi(f_{1}),\Phi(g_{1}))=(\Phi(g_{1}^{\prime}),\Phi(f_{1}^{\prime}))$ 
    so that we get the required $\theta$ compatibility to ensure that the triples 
    $(\mathcal{G}_1,\mathcal{H}_1,\theta_1)$ and 
    $(\mathcal{G}_2,\mathcal{H}_2,\theta_2)$ are equivalent.
\end{proof}

{\bf Acknowledgement.}
The first author is supported by SERB MATRICS grant no. MTR/2022/000515 (Govt. of India). The second author is partially supported by  JC Bose Fellowship and grant of Prof Debashish Goswami,  ISI Kolkata,  awarded
to him by SERB, DST (Govt of India). The authors would like to thank Dr. Arnab Mandal for some useful discussions. The authors would also like to thank the anonymous referee for their useful comments.


\begin{bibdiv}
   \begin{biblist}
    \bib{Qinfo2}{article}{
   author={Asterias A.},
   author={M\v{a}ncinska L.},
   author={Roberson David E.},
   author={\v{S}ámal R.},
   author={Severini S.},
   author={Varvitsiotis A},
   title={Quantum and non-signalling graph isomorphisms}, journal={J. Combin. Theory Ser. B}, 
   volume={136}, 
   date={2019}, 
   pages={289-328}
   }
     
    \bib{Banica}{article}{
   author={Banica, T.},
   title={Quantum automorphism groups of small metric spaces}, journal={Pacific J. Math}, 
   volume={219(1)}, 
   date={2005}, 
   pages={27-51}
   }
   \bib{Skalski}{article}{
   author={Banica T.},
   author={Skalski A.},
   title={Quantum symmetry groups of $C^{\ast}$-algebras equipped with an orthoginal filtration}, 
   journal={Proc. London Math. Soc.}, volume={106}, date={2013}, pages={980-1004}
   }
   \bib{DeCom}{article}{
   author={Baum Paul F.},
   author={De Commer, K.},
   author={ Hajac, Piotr M.}, 
   title={Free actions of compact quantum groups on unital $C^{\ast}$-algebras},
   journal={Documenta Math.},
   volume={22}, 
   date={2017}, 
   pages={825-849}
   }
   \bib{Joardar2}{article}{
   author={Bhattacharjee, S.},
   author={Joardar, S.},
   title={Equivariant $C^{\ast}$-correspondence and compact quantum group actions on Pimsner algebras}, 
   journal={https://doi.org/10.48550/arXiv.2209.0470}
   }
   \bib{Bichon}{article}{
   author={Bichon, J.},
   title={Quantum automorphism groups of finite graphs}, 
   journal={Proc. Am. Math. Soc.}, 
   volume={131(3)}, 
   date={2003}, 
   pages={665-673}
   }
    \bib{bichon2004free}{article}{
   author={Bichon, J.},
   title={Free wreath product by the quantum permutation group}, 
   journal={Algebr. Represent. Theory}, 
   volume={7}, 
   date={2004, no 4}, 
   pages={343-362}
   }
   \bib{Fulton}{article}{
   author={Fulton},
   title={The quantum automorphism group and directed trees},
   journal={Ph. D. thesis, Virginia}, 
   date={2006}
   } 
   \bib{Hajac}{article}{
   author={Gardella, E.},
   author={Hajac, Piotr M.}, 
   author={Tobolski, M.}, 
   author={Wu, J.},   
   title={The local triviality dimensions of actions of compact quantum groups}, 
   journal={https://doi.org/10.48550/arXiv.1801.00767}
   } 
    \bib{Goswami}{article}{
   author={Goswami, D.},
   title={Quantum group of isometries in classical and noncommutative geometry}, 
   journal={Comm. Math. Phys.}, 
   volume={285(1)}, 
   date={2009}, 
   pages={141-160}
   }
   \bib{Asfaq}{article}{
   author={Goswami, D.},
   author={Asfaq Hossain, Sk.},   
   title={Quantum symmetry in multigraphs}, 
   journal={https://doi.org/10.48550/arXiv.2302.08726}
   }
   \bib{Potts}{article}{
   author={Goswami, D.},
   author={Asfaq Hossain, Sk.},   
   title={Quantum symmetry on Potts Model}, 
   journal={J. Math. Phys.},
   volume={63, no 4},
   date={2022},
   pages={14 pages}}
   \bib{Hazzlewood1}{article}{
   author={Hazzlewood, R.},
   title={Constructing $k$-graphs from $k$-coloured graphs}, 
   journal={Unpublished Honours thesis}
   }
   \bib{Sims}{article}{
   author={Hazzlewood, R.},
   author={Raeburn, I.}, 
   author={Sims, A.}, 
   author={Webster, S.B.G.},    
   title={Remarks on some fundamental results about higher rank graphs and their $C^{\ast}$-algebras}, 
   journal={Proc. of the Edin. Math. Soc. (2)}, volume={56}, date={2013},pages={575-597}
   }
   \bib{edin}{article}{
    author={Joardar, S.}, 
    author={Mandal, A.},
    title={Invariance of KMS states on graph $C^{\ast}$-algebras under classical and quantum symmetry}, 
    journal={Proc. of the Edin. Math. Soc. (2)}, 
    volume={64}, 
    date={2021}, 
    pages={762-778}
    }
    \bib{Joardar1}{article}{
    author={Joardar, S.},
    author={Mandal, A.},
    title={Quantum symmetry of graph $C^{\ast}$-algebras at critical inverse temperature}, 
    journal={Studia Mathematica},
    year={2018},
    volume={256},
    pages={1-20}
    }
   \bib{Qinfo}{article}{
    author={Lupini, M.},
    author={Mancinska, L.},
    author={Robertson, David E.},
    title={Nonlocal games and quantum permutation groups}, 
    journal={J. Funct. Anal.},
    year={2020},
    volume={279, no 5},
    pages={44 pages}
    }
   \bib{Sims2}{article}{
   author={Pask, D.},
   author={Raeburn, I.},
   author={Rordam, M.},
   author={Sims, A.},   
   title={Rank two graphs whose $C^{\ast}$-algebras are direct limits of circle algebras}, 
   journal={J. Funct. Anal.}, 
   volume={239}, 
   date={2006, no 1}, 
   pages={137-178}
   }
   \bib{NYJM}{article}{
    author={Pask, D.},
    author={Kumjian, A.},
    title={Higher rank graph $C^{\ast}$-algebras}, 
    journal={New York J. Math.}, 
    volume={6}, 
    date={2020}, 
    pages={1-20}
    }
    \bib{Sims1}{article}{
        author={Raeburn, I.},
        author={Sims, A.},
    author={Yeend, T.},
       title={Higher rank graphs and their $C^{\ast}$-algebras}, 
    journal={Proc. Edin. Math. Soc. (2)}, 
    volume={46}, 
    date={2003},
    number={1}, 
    pages={99-115}
    }
    \bib{Steger1}{article}{
    author={Robertson, G.},
    author={Steger, T.},
        title={Affine buildings, tiling systems and higher rank Cuntz-Krieger algebras}, 
    journal={J. Reine Angew}, 
    volume={Math. 513}, 
    date={1999}, 
    pages={115-144}
    }
    \bib{Steger2}{article}{
    author={Robertson, G.},
    author={Steger, T.},    
    title={$K$-theory for rank $2$ Cuntz-Kriger algebras}, 
    journal={Preprint}
    }
    \bib{Vaes}{article}{
    author={Rollier, L.},
    author={Vaes, S.},
        title={Quantum automorphism groups of connected locally finite graphs and quantizations of discrete groups}, 
    journal = {International Mathematics Research Notices},
     year = {2023},
    pages = {rnad099},
    issn = {1073-7928},
    }
    \bib{Schmidt}{article}{
    author={Schmidt, S.},
    title={Quantum automorphism groups of finite graphs},
    journal={Thesis}
    }
    \bib{Weber}{article}{
            author={Schmidt, S.},
    author={Weber, M.},
        title={Quantum symmetries of graph $C^{\ast}$-algebras}, 
    journal={Canad. Math. Bull.}, 
    volume={61}, 
    date={2017}, 
    pages={848-864}
    }
    \bib{Soltan}{article}{
    author={M. Soltan, Piotr},
    title={On actions of compact quantum groups}, 
    journal={Illinois J. Math.}, 
    volume={55}, 
    date={2011}, 
    pages={953-962}
    }
    
   \bib{Voigt2}{article}{
    author={Voigt, C.},
    title = {Infinite quantum permutations},
    journal = {Advances in Mathematics},
    volume = {415},
    year = {2023},
    pages = {108887}
    } 

    \bib{voigtbc}{article}{
    author={Voigt, C.},
    title={The Baum-Connes conjecture for free orthogonal quantum groups},
    journal={Adv. Math.},
    volume={227},
    date={2011},
    number={5},
    pages={1873--1913}
    }                 
    \bib{Wang}{article}{
		author={Wang, S.},
		title={Quantum symmetry groups of finite spaces},
		journal={Comm. Math. Phys.},
		volume={195},
		date={1998},
		number={1},
		pages={195--211},
		issn={0010-3616},
		review={\MR{1637425}},
		doi={10.1007/s002200050385},
	 }
  
       	\bib{freewang}{article}{
		author={Wang, S.},
		title={Free products of compact quantum groups},
		journal={Comm. Math. Phys.},
		volume={167},
		date={1995},
		number={3},
		pages={671-692},
  }
	 \bib{Woro}{article}{
		author={Woronowicz, S.L.},
		title={Compact matrix pseudogroups},
		journal={Comm. Math. Phys.},
		volume={111},
		date={1987},
		number={4},
		pages={613--665},
		issn={0010-3616}
	 }
  
     \end{biblist}
  \end{bibdiv}
\end{document}